\documentclass[12pt,draftcls,onecolumn,romanappendices]{IEEEtran}
\usepackage{amsfonts,amssymb}
\usepackage[cmex10]{amsmath}
\usepackage{epsfig}
\usepackage{amsfonts}
\usepackage{graphics}
\usepackage{graphicx}
\usepackage{subfig}
\usepackage{mathtools}
\usepackage{enumerate}
\usepackage{cite}
\usepackage{algorithm}
\usepackage{algorithmic}
\usepackage{psfrag}
\usepackage{comment}
\usepackage{color}
\epsfverbosetrue

    \def\independenT#1#2{\mathrel{\setbox0\hbox{$#1#2$}%
    \copy0\kern-\wd0\mkern4mu\box0}}

\newcommand{\el}{{ ~~\it Example: \ }}
\newcommand{\mat}[1]{\boldsymbol{#1}}

\title{Optimal Server Assignment in Multi-Server Queueing Systems with Random Connectivities}  

\author{Hassan~Halabian,~\IEEEmembership{Student Member,~IEEE,}
        Ioannis~Lambadaris,~\IEEEmembership{Member,~IEEE,}
        Yannis Viniotis,
        and~Chung-Horng~Lung,~\IEEEmembership{Member,~IEEE,}
\thanks{H. Halabian, I. Lambadaris, C-H Lung are with the Department of Systems and
Computer Engineering, Carleton University, Ottawa, ON, K1S 5B6 Canada (e-mail:
hassanh@sce.carleton.ca; ioannis@sce.carleton.ca; chlung@sce.carleton.ca).

Y. Viniotis is with the Department of Electrical and Computer Engineering,
North Carolina State University, Raleigh, North Carolina (e-mail:
candice@ncsu.edu).

}\vspace{-2mm}
}
\begin{document}             
\maketitle                   
\newtheorem{theor}{Theorem}
\newtheorem{lem}{Lemma}
\newtheorem{pro}{Proposition}
\newtheorem{defin}{Definition}
\newtheorem{conj}{Conjecture}
\newtheorem{cor}{Corollary}

\IEEEpeerreviewmaketitle
\begin{abstract}
We study the problem of assigning $K$ identical servers to a set of $N$ parallel queues in a time-slotted queueing system. The \textit{connectivity} of each queue to each server is randomly changing with time; each server can serve \textit{at most one} queue and each queue can be served by \textit{at most one} server during each time slot. Such a queueing model has been used in addressing resource allocation  problems in wireless networks. It has been previously proven that Maximum Weighted Matching (MWM) is a \textit{throughput-optimal} server assignment policy for such a queueing system. In this paper, we prove that for a system with i.i.d. Bernoulli packet arrivals and connectivities, MWM minimizes, in stochastic ordering sense, a broad range of cost functions of the queue lengths such as total queue occupancy (which implies minimization of average queueing delays). Then, we extend the model by considering imperfect services where it is assumed that the service of a scheduled packet fails randomly with a certain probability. We prove that the same policy is still optimal for the extended model. We finally show that the results are still valid for more general connectivity and arrival processes which follow conditional permutation invariant distributions.  


 
\vspace{-2mm}
\end{abstract}

\section{Introduction}
\IEEEPARstart{O}{ptimal} stochastic control is one of the main objectives in the design of emerging wireless networks. One of the primary goals in stochastic control and optimization of wireless networks is to distribute the shared resources in the physical (e.g., power) and MAC layers (e.g., radio interfaces, relay stations and orthogonal sub-channels) among multiple users such that certain stochastic performance attributes are optimized.
While various performance criteria including the stable throughput region, power consumption and utility functions of the admitted traffic rates have been studied in several papers \cite{Tassiulas92,Tassiulas93,sarkar,Now,mckeown,Tassiulas97,Tassiulas98,Leonardi,Lin06atutorial,bambos2002,Tsibonis2004, Luo1999,Pappas2010, Naware2005, hassan_isit10,hassan_isit11,hassan-milcom2010,v3,v4}, average queueing delay has received less attention.
The inherent randomness in wireless channels makes delay-optimal resource allocation a challenging problem in wireless networks.

In this paper, we focus   on delay-optimal server assignment in a time-slotted, multi-queue, multi-server system with random connectivities. Random connectivities can model unreliable and randomly varying wireless channels. Our queueing model can be applied to study resource allocation in wireless access networks where the wireless users are modeled by the queues; the shared resources are modeled by the servers and the wireless channels are modeled by the random connectivities between the queues and the servers.
Although this model is a simplified representation of a real wireless system, nevertheless it does provide valuable intuition for the performance optimization of real systems. Similar modeling approaches have already appeared in \cite{Tassiulas93,bambos95,bambos2002,Koole2001,javidi2,hassan_isit10,hassan_isit11,ganti,sarkar,hassan-milcom2010}.

\vspace{-2mm}
\subsection{Related Work and Our Contributions}\label{relatedwork}

The problem of \textit{throughput-optimal} server allocation in \textit{multi-queue, single-server} systems with \textit{random connectivities} was addressed in \cite{Tassiulas93,bambos95,bambos2002,Koole2001}. 
In \cite{Tassiulas93}, the authors considered a time-slotted, multi-queue single-server system with Bernoulli packet arrivals and connectivities from each of the queues to a single server. They introduced LCQ (Longest Connected Queue) policy as a throughput-optimal policy and also characterized the stability region by a set of linear inequalities.
The authors in \cite{bambos95} considered a continuous-time version of the model studied in \cite{Tassiulas93} with finite buffer space and showed that under stationary ergodic input job flow and modulation
processes, LCQ policy maximizes the stable throughput region of this system. In \cite{bambos2002}, C-FES (Connected queue with the Fewest Empty Spaces) policy, a policy that allocates the server to the connected queue with the
fewest empty spaces, was introduced for this system. It was shown that C-FES stochastically minimizes the loss flow and maximizes the throughput of the system.
In \cite{Koole2001}, a model similar to the model of \cite{bambos2002,Tassiulas93} was studied   and it was shown that the Best User (BU) policy maximizes the expected discounted number of successful transmissions.



While in throughput-optimal server allocation the objective is to find a policy that maximizes the throughput region of the system and keeps the queues stable \cite{Tassiulas92,Now}, in \textit{delay-optimal} server allocation the goal is to determine a policy that minimizes the average queueing delay. Thus, the objective in delay optimality is more stringent than the objective in throughput optimality. A server allocation policy may be throughput-optimal but not delay-optimal; however, a delay-optimal policy (for all the arrival rates) is always throughput-optimal. 
In \cite{Tassiulas93}, the authors (other than proving the throughput optimality of LCQ as mentioned earlier) proved that for a \textit{multi-queue, single-server} system with i.i.d. Bernoulli arrival and connectivity processes, the LCQ policy 
is also delay-optimal. The extension of this result for non-i.i.d. case is still an open problem.

In generalizing the results to \textit{multi-queue, multi-server} (MQMS) systems, various multi-server systems have been studied \cite{javidi2,hassan_isit10,hassan_isit11,sarkar,ganti,javidi,hussein}. 
In \cite{javidi2}, Maximum Weight (MW) policy was proposed as a throughput-optimal server allocation policy for an MQMS queueing system with general, stationary channel processes. MW policy can be considered as a special case of back-pressure algorithm which was proven in \cite{Tassiulas92,Now} to be a throughput-optimal resource allocation algorithm in a general queueing system.
In \cite{hassan_isit10}, the authors characterized the network stability region of multi-queue, multi-server systems with \textit{time-varying, independent connectivities}. 
The results were further extended in \cite{hassan_isit11} for more general, stationary channel distributions (and not just independent Bernoulli channels). 
In all the models studied in \cite{javidi2,hassan_isit10,hassan_isit11}, there is no restriction on the number of servers that can be allocated to a queue. For ease of reference, we will call such an MQMS system as MQMS-Type1.
In \cite{sarkar}, it was shown that for an MQMS system in which \textit{the queues are restricted to get service from at most one server during each time slot}, Maximum Weighted Matching (MWM) policy is throughput-optimal. For ease of reference, we will call such an MQMS system with this extra assumption as MQMS-type2.  The authors also considered the effect of infrequent channel state measurements on the network stability region of MQMS systems. Similar to MQMS-Type1, for MQMS-Type2 the throughput-optimal policy (MWM) can be considered as a special case of the back-pressure algorithm.

In contrast to the single-server system (where LCQ was both throughput-optimal and delay-optimal), in MQMS-Type1 system the MW policy is not necessarily delay-optimal. More specifically, in \cite{hassan_isit10} it was also shown that although MW policy is throughput-optimal, even for a system with i.i.d. Bernoulli arrivals and connectivity processes, MW policy in its general form, is not delay-optimal.





The delay-optimal server allocation problem in multi-server systems was addressed in \cite{ganti,javidi,hussein}. The authors in \cite{ganti} considered a queueing model with a set of parallel queues and i.i.d. Bernoulli packet arrivals that are competing to attract service from $K$ identical servers forming a \textit{server-bank}. The connectivities of the queues to the \textit{entire} server-bank are assumed to be i.i.d. Bernoulli processes. Each queue is restricted to receive service from \textit{at most one server during each time slot}. The authors proposed LCQ policy in which the servers of the server-bank are allocated to the $K$ longest connected queues at each time slot. Using dynamic coupling and stochastic ordering, they proved the delay optimality of LCQ policy for such a system. In our work, the focus would be on delay optimality of MWM policy in MQMS-Type2 system in which the servers are not restricted to form a server-bank. Instead, we assume that each queue has an independent connectivity to each individual server (as seen in Figure \ref{mot}).
The work in \cite{javidi,hussein} focuses on delay optimal server allocation problem in the MQMS-Type1 system. In \cite{javidi}, the authors
introduced MTLB (Maximum-Throughput Load-Balancing) policy and using dynamic programming showed that this policy minimizes a class of cost functions including total average delay for the case of \textit{two queues} with  i.i.d., Bernoulli-distributed arrivals and connectivities. In \cite{javidi}, no general argument was provided for the optimality of MTLB for more than two queues.
The work in \cite{hussein} considers this problem for a general number of queues and servers. In \cite{hussein}, a class of \textit{Most Balancing} (MB) policies was characterized among all work-conserving policies which  minimize, in stochastic ordering sense, a class of cost functions including total queue occupancy (and thus are delay-optimal). However, this class of proposed MB policies is just \textit{characterized by a property of this class}; the authors did not introduce an \textit{explicit implementation} for the optimal policy.
In this paper, we focus on MWM policy and prove that this throughput-optimal policy \textit{is also delay-optimal} for an MQMS-Type2 system with i.i.d. arrival and connectivity processes. 
Our work extends the results derived in \cite{Tassiulas93,ganti}. In particular, the researchers in \cite{Tassiulas93,ganti} have considered queueing models where a \textit{single server} or a \textit{server-bank} is randomly connected to a set of parallel queues. In this paper, we consider a more general model where \textit{each individual server is randomly connected} to the queues (as seen in Figure \ref{mot}). 
Although the two models bear certain similarities, extending the results from single-server (server-bank) system to multi-server system is not a straightforward procedure. Our work is different from the work in \cite{Tassiulas93,ganti} from both the modeling power and the difficulty in proof points of view.
%
%
%

%

For more information on optimal scheduling and resource allocation problems in wireless networks the reader is encouraged to also consult with \cite{Now,Liu2003,Agrawal2002,atila,Stolyar2005,Lu1999}.




\begin{figure}[tp]
    \centering
    \psfrag{1}[][][.9]{$1$}
    \psfrag{2}[][][.9]{$2$}
    \psfrag{K}[][][.9]{$K$}
    \psfrag{j}[][][.9]{$\lambda$}
    \psfrag{f}[][][.9]{a) The model studied in \cite{Tassiulas93,ganti}}  
    \psfrag{r}[][][.9]{b) MQMS-Type2 system in this paper.}
    \psfrag{a}[][][1]{$p$}
    \psfrag{g}[][][1]{$X_1(t)$}
    \psfrag{h}[][][1]{$X_2(t)$}
    \psfrag{i}[][][1]{$X_N(t)$}
    \psfrag{k}[][][.9]{Single server or server-bank}
    \includegraphics[width=0.96\textwidth]{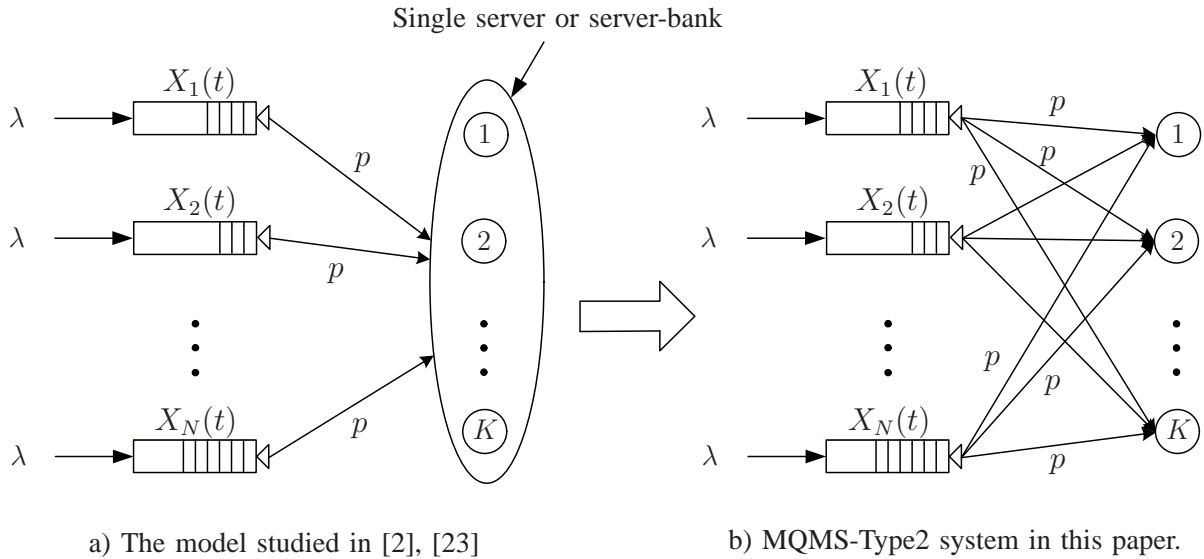}
	\caption{Previous models vs. our model. ($\lambda$ is the arrival probability and $p$ is the connectivity probability)}\vspace{-7mm}
	\label{mot}
\end{figure}


Our contributions in this paper are summarized as follows: 
First, for an MQMS-Type2 system 
we prove that during each time slot, Maximum Weighted Matching (MWM) policy will result in the most balanced queue vector in the system, i.e., \textit{maximization of the matching weight and balancing of the queues are equivalent}. Graph theoretic arguments were applied to prove this result that is formally introduced in Lemma \ref{l1} and Lemma \ref{l2} later in the paper. Note that our approach to prove this result is only applicable to the MQMS-Type2 model (due to the structure of the model and the MWM policy) and cannot be easily extended to MQMS-Type1 system.
Second, using this result in conjunction with the notions of stochastic ordering and dynamic coupling, we prove the delay optimality of MWM policy for an MQMS-Type2 system with i.i.d. Bernoulli arrivals and connectivities.
More specifically, we prove that MWM minimizes, in \textit{stochastic ordering} sense, a range of cost functions of queue lengths including total queue occupancy\footnote{The optimality of MWM is proven among all causal server assignment policies.}.
Third, we then extend our model by considering imperfect services where it is assumed that the service of a scheduled packet fails randomly with a certain probability. We prove that MWM is still optimal for the extended model. We finally show that the results are still valid for some more general connectivity and arrival processes which follow conditional permutation invariant distributions.

The rest of this paper is organized as follows. 
In Section \ref{modeldes}, we introduce the queueing model and the required notation. In Section \ref{background}, we describe the Maximum Weighted Matching (MWM) server assignment policy.
In Section \ref{asli}, we prove 
the delay optimality of MWM server assignment policy. 
In Section \ref{sim}, we present simulation results where we compare the performance of MWM policy with the performance of two other server assignment policies in terms of average total queue occupancy (or equivalently average queueing delay). Finally, we summarize our conclusions in Section \ref{conc}.



\section{Model Description}
\label{modeldes}

Throughout the paper, random variables are represented by CAPITAL letters and lower case letters are used to represent sample values of the random variables. Moreover, we use boldface font to represent matrices and vectors.

We consider a time-slotted, MQMS-Type2 system consisting of a set of parallel
queues $\mathcal{N}=\{1,2,\dots,N\}$ with infinite buffer space for each queue (see Figure \ref{model}). 
Packets in this system are assumed to have constant length and require one time slot to complete service. The service to this set of queues is provided by a set of identical servers $\mathcal{K}=\{1,2,\dots,K\}$. 
The connectivity of each queue $n\in \cal N$ to each server $k \in \cal K$ at each time slot $t$ is random and varying across time slots. We denote the connectivity of queue $n$ to server $k$ at time slot $t$ by $C_{n,k}(t) \in \{0,1\}$. 
When $C_{n,k}(t)=1$ ($C_{n,k}(t)=0$), queue $n$ is connected to (disconnected from) server $k$ at time slot $t$.
The connectivity variables $C_{n,k}(t)$ are assumed to be i.i.d. Bernoulli random variables with a fixed parameter $p$\footnote{The actual value of $p$ does not involve in our analysis. We only rely on the fact that the connectivities are i.i.d. Bernoulli processes.}.
 
At any time slot, each server can serve at most one packet from a connected, non-empty queue.
We do not allow server sharing in the system, i.e., a server can serve at most one queue per time slot. 
We also assume that \textit{at most one server} can be assigned to any connected queue during a time slot.
 
\begin{figure}[tp]
    \centering
    \psfrag{a}[][][.9]{$C_{1,1}(t)$}
    \psfrag{b}[][][.9]{$C_{1,2}(t)$}
    \psfrag{c}[][][.9]{$C_{1,K}(t)$}
    \psfrag{d}[][][.9]{$C_{N,1}(t)$}
    \psfrag{e}[][][.9]{$C_{N,2}(t)$}
    \psfrag{f}[][][.9]{$C_{N,K}(t)$}  
    \psfrag{g}[][][1]{$X_1(t)$}
    \psfrag{h}[][][1]{$X_2(t)$}
    \psfrag{i}[][][1]{$X_N(t)$}
    \psfrag{j}[][][1]{$A_1(t)$}
    \psfrag{k}[][][1]{$A_2(t)$}
    \psfrag{l}[][][1]{$A_N(t)$}
    \includegraphics[width=0.45\textwidth]{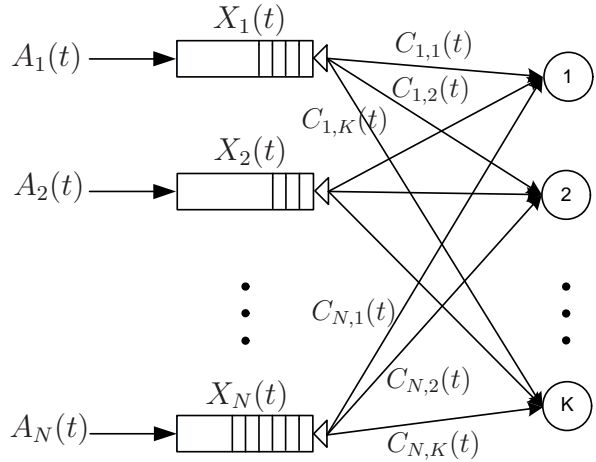}
	\caption{ Discrete-time MQMS-Type2 system with $N$ parallel queues and $K$ servers.}\vspace{-7mm}
	\label{model}
\end{figure}
%

Let $A_n(t)$ denote the number of packet arrivals to queue $n$ at time slot $t$. We assume that new arrivals at each time slot are added to the queues at the end of the time slot. The arrival variables $A_n(t)$ are assumed to be i.i.d. Bernoulli random variables with the same parameter $\lambda$ for all $n$ and $t$\footnote{The actual value of $\lambda$ does not involve in our analysis. We only rely on the fact that the arrivals are i.i.d. Bernoulli processes.}.
 
We denote the length of queue $n$ at the end of time slot $t$ (i.e., after adding the new arrivals) by $X_n(t)$. Hence, $X_n(t)$ represents the number of packets in the $n$th queue at the end of time slot $t$ (or beginning of time slot $t+1$). 




\subsection{Server Assignment Policy}
At each time slot $t$ the server assignment policy has to decide about a \textit{bipartite (graph) matching}\footnote{A matching in a bipartite graph is a sub-graph of the original graph in which no two edges share a common vertex.} between sets $\cal N$ and $\cal K$. We assume that this decision is made in a causal fashion, i.e., based on the available history of arrival processes, service processes, queue states and the connectivity states until time $t$.

A policy $\pi$ is fully determined by its indicator variables $M^{(\pi)}_{n,k}(t)~\forall n\in \mathcal{N}, \forall k\in \mathcal{K}, t=1,2,\dots$ which are defined as
\begin{eqnarray}
M^{(\pi)}_{n,k}(t) = \begin{cases} 1, & \mbox{if server $k$ is assigned to queue $n$ by policy $\pi$ at time slot $t$,} \\ 0, & \mbox{otherwise. }  \end{cases}
\end{eqnarray}
 
We define the $N\times K$ matrix $\mat{M}^{(\pi)}(t)=(M^{(\pi)}_{n,k}(t)), \forall n \in \mathcal{N},\forall k\in \mathcal{K}$ as the \textit{employed matching by policy $\pi$} at time slot $t$.
Hence, a server assignment policy $\pi$ can be defined as the set of all the employed matchings by policy $\pi$ at time slots $t=1,2,\dots$, i.e., $\pi=\{\mat{M}^{(\pi)}(t)\}_{t=1}^{\infty}$.
We denote the \textit{matching space} containing all the possible assignments of the servers to the queues by $\cal M$. 
The set $\mathcal{M}$ is equivalent to the set of all the possible matchings in an $N\times K$ \textit{complete bipartite graph}\footnote{A complete bipartite graph is a bipartite graph in which each vertex in each part is connected to all the vertices in the other part. An $N\times K$ bipartite graph has $NK$ edges.}.

We can observe that $X_n(t)$, the queue length random variable, evolves in time as follows:
\begin{eqnarray} \label{evolution}
X_n(t)= \left (X_n(t-1)-\displaystyle\sum_{k=1}^{K}C_{n,k}(t)M_{n,k}^{(\pi)}(t) \right)^+ + A_n(t)~~~~~~\forall n \in \cal N
\end{eqnarray}

The operator $(\cdot)^+$ returns the term inside the parentheses if it is non-negative and zero otherwise. 

The queueing model introduced in this section is useful in providing intuition for modeling  resource assignment problems in various systems with shared resources \cite{hassan-milcom2010,sarkar}. In wireless communication systems, resources such as communication sub-channels, relay stations, etc. are shared among users.
As an example, we can consider a relaying access network with $N$ users and $K$ shared relays. By modeling the cooperative wireless channel between each user, each relay and the base station as an erasure channel, the performance of such a system can be studied following our model in Figure \ref{model}.

 
\vspace{-2mm}
\section{Maximum Weighted Matching (MWM) Server Assignment Policy}  \label{background}




\subsection{MWM Optimization Problem}
In \cite{Tassiulas92,Now}, it was shown that \textit{back-pressure} algorithm maximizes the stable throughput region of a general data network, i.e., it is throughput-optimal. The reader may refer to \cite{Tassiulas92,Now} for more information about back-pressure algorithm. For the model introduced in Section \ref{modeldes}, the back-pressure algorithm reduces to the following optimization problem at each time slot $t$ \cite{sarkar}. In this integer programming problem, $M_{n,k}(t)$ variables are the optimization variables and $X_n(t-1) $ and $C_{n,k}(t)$ are known parameters.
\begin{eqnarray}\label{matching}
\texttt{Maximize:}_{\substack{\\ \\ \hspace{-2.2cm}\text{$M_{n,k}(t), \forall n,k$}}}&~ \displaystyle\sum_{n=1}^{N} X_n(t-1)\displaystyle\sum_{k=1}^{K} M_{n,k}(t)C_{n,k}(t)~~~~~~~ \nonumber \\
\texttt{Subject to:}&\displaystyle\sum_{k=1}^{K} M_{n,k}(t)\leq 1 ~~ (n=1,2, \dots ,N),~~ \nonumber\\
&\displaystyle\sum_{n=1}^{N} M_{n,k}(t)\leq 1 ~~ (k=1,2,\dots ,K), ~~
\nonumber\\
&~~~~~~~~~~~~~~~~~~~M_{n,k}(t)\in \{0,1\} ~~~ (k=1,2, \dots ,K),(n=1,2, \dots ,N)
\end{eqnarray} 
Finding the solution of problem (\ref{matching}) is equivalent to finding a maximum weighted matching in the $N \times K$ bipartite graph $G_t=(\mathcal{N}, \mathcal{K},\mathcal{E})$ shown in Figure \ref{matching-graph}. 
Hence, the back-pressure algorithm for the queueing model of Figure \ref{model} is also known as 
Maximum Weighted Matching (MWM) algorithm. In $G_t$, $\mathcal{N}$ and $\mathcal{K}$ are the two sets of vertices in each part of the graph and $\mathcal{E}=\{e_{n,k},\forall n \in \mathcal{N}, \forall k \in \mathcal{K}\}$ is the set of edges between these two parts. 
In $G_t$, the associated weight to each edge $e_{n,k}$ is $ X_n(t-1)C_{n,k}(t) $. A matching in graph $G_t$ is a sub-graph of $G_t$ in which no two edges share a common vertex. Any matching $\mat{M}^{(\pi)}(t)$ at any time slot $t$ is corresponding to a sub-graph of $G_t$ namely $G_t^{(\pi)}=(\mathcal{N},\mathcal{K}, \mathcal{E}^{(\pi)})$ in which  $e_{n,k} \in \mathcal{E}^{(\pi)}$ if and only if $M^{(\pi)}_{n,k}(t)=1$.
There are several algorithms to find the maximum weighted matching in bipartite graphs. The most well-known one is the Hungarian algorithm with    $O((\min\{N,K\})(\max\{N,K\})^2)$ complexity \cite{hungarian}.
\subsection{MWM Policy}
Assume that $\mat{M}^{(\text{MWM})}(t)=(M^{(\text{MWM})}_{n,k}(t))~ \forall n \in \mathcal{N},\forall k\in \mathcal{K}$ is the matching whose indicator variables are the solution of the optimization problem (\ref{matching}). $\mat{M}^{(\text{MWM})}(t)$ has the following properties: 
\begin{itemize}
\item[\textit{(a)}] $\mat{M}^{(\text{MWM})}(t)$ always exists at all time slots.
\item[\textit{(b)}]  The maximum weighted matching in a bipartite graph \textit{may not be unique}, i.e., there may be more than one matching $\mat{M}^{(\text{MWM})}(t)$ for the graph of Figure \ref{matching-graph} at each time slot.
\end{itemize}

\begin{figure}[tp]
    \centering
    \psfrag{a}[][][.9]{$1$}
    \psfrag{b}[][][.9]{$2$}
    \psfrag{N}[][][.9]{$N$}
    \psfrag{c}[][][.9]{$1$}
    \psfrag{d}[][][.9]{$2$}
    \psfrag{K}[][][.9]{$K$}  
    \psfrag{e}[][][.85]{$X_1(t-1)C_{1,1}(t)$}
    \psfrag{f}[][][.85]{$X_1(t-1)C_{1,2}(t)$}
    \psfrag{g}[][][.85]{$X_2(t-1)C_{2,2}(t)$}
    \psfrag{h}[][][.85]{$X_2(t-1)C_{2,K}(t)$}
    \psfrag{i}[][][.85]{$X_N(t-1)C_{N,1}(t)$}
    \psfrag{j}[][][.85]{$X_N(t-1)C_{N,2}(t)$}
    \psfrag{l}[][][.85]{$X_N(t-1)C_{N,K}(t)$}
    \includegraphics[width=.4\textwidth]{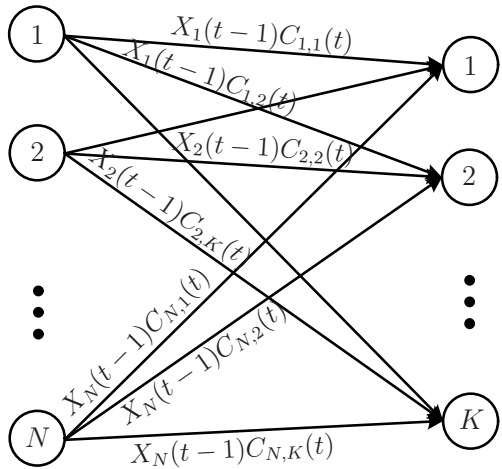}
    \caption{Bipartite graph for the Maximum Weighted Matching (MWM) policy.} \vspace{-7mm}
    \label{matching-graph}
\end{figure}




\begin{defin}
A Maximum Weighted Matching (MWM) server assignment policy is defined as a policy that employs maximum weighted matching $\mat{M}^{(\text{MWM})}(t)$ at all time slots, i.e., $\pi^{(\text{MWM})}=\{\mat{M}^{(\text{MWM})}(t)\}_{t=1}^{\infty}$.
\end{defin}
An MWM policy at each time slot observes the queue lengths $X_n(t-1)$ and the connectivity variables $C_{n,k}(t)$ and determines a maximum weighted matching (the matching indicator variables) in the bipartite graph of Figure \ref{matching-graph}. Note that, by construction, the MWM  policy is causal. 


\begin{defin}
We denote the set of all policies that employ maximum weighted matching at all time slots by $\Pi^\text{MWM}$. 
\end{defin}

According to  property  \textit{(a)}  above, the set $\Pi^\text{MWM}$ is not empty. Moreover, according to property \textit{(b)}, we conclude that $\Pi^\text{MWM}$ may contain an infinite number of policies.



\section{Delay Optimality of MWM Policy}   \label{asli}  


In this section, we  prove the delay optimality of an MWM policy $\pi \in \Pi^{\text{MWM}}$. This result is formally presented in Theorem \ref{th2}. 
More specifically, we show that in an MQMS-Type2 system with i.i.d. Bernoulli arrival and connectivity processes, any MWM policy is optimal in minimizing, \textit{in stochastic ordering sense}, a class of cost functions of queue length processes including average queueing delay. For brevity we will use the term ``\textit{delay optimality}'' to refer to the optimality of MWM in this sense. 


\vspace{-3mm}
\subsection{Delay Optimality of MWM Policy-Outline of the Proof}
\noindent The proof of Theorem \ref{th2} proceeds along the following steps:

\begin{figure}[tp]
    \centering
    \psfrag{b}[][][.9]{$\mathsf{MW}$ index}
    \psfrag{x}[][][.9]{Maximization}
    \psfrag{a}[][][.9]{Queue}
    \psfrag{n}[][][.9]{Balancing}
    \psfrag{c}[][][.75]{Lemma \ref{l1}}
    \psfrag{d}[][][.75]{Lemma \ref{l2}}
    \psfrag{f}[][][.84]{Policy Improvement}
    \psfrag{y}[][][.84]{(Delay Reduction)}
    \psfrag{g}[][][.75]{Lemma \ref{l3}}
    \psfrag{h}[][][.9]{for any $\pi \notin \Pi^{\text{MWM}}$}
    \psfrag{i}[][][.75]{Theorem \ref{th1}}
    \psfrag{j}[][][.9]{Delay-optimal policy}
    \psfrag{z}[][][.9]{belongs to $\Pi^{\text{MWM}}$}
    \psfrag{k}[][][.75]{All MWM policies result}
    \psfrag{p}[][][.75]{in the same average delay}
    \psfrag{l}[][][.8]{Lemmas \ref{lnew} and \ref{lperm}}
    \psfrag{m}[][][.8]{Any MWM policy}
    \psfrag{o}[][][.8]{is delay optimal}
    \includegraphics[width= 1\textwidth]{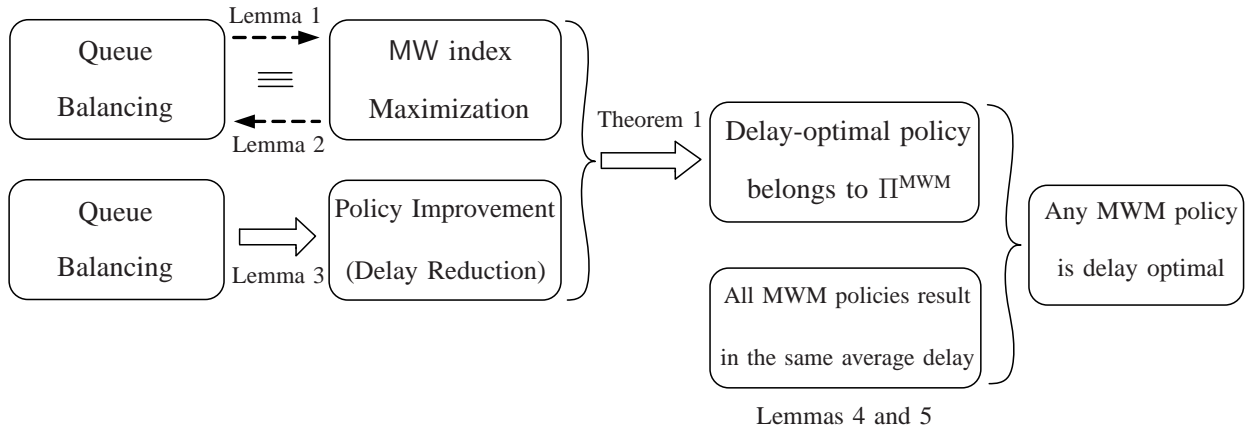}\vspace{-2mm}
    \caption{Outline of the proof} \vspace{-6mm}
    \label{proofroadmap}
\end{figure}

First, We will introduce the notion of \textit{balanced queue vectors} and the corresponding \textit{balancing server reallocation} at time slot $t$ in Definition \ref{balancing}. For any given  policy $\pi$ and a fixed time slot $t$, we will also define the \textit{Matching Weight} index $\mathsf{MW}_{\pi}(t)$ in Definition \ref{mindex}. Note that this index is not directly related to (average) delay; it is, however, a crucial link in comparing arbitrary policies to the MWM ones in the set $\Pi^\text{MWM}$ defined in the previous section. 
We show then in Lemmas \ref{l1} and \ref{l2} that the notions of ``maximizing the Matching Weight index" and ``producing balanced queue vectors" via balancing server reallocations are equivalent. This property allows us to characterize MWM policies as ones that produce the most balanced queue size vectors possible.

Second, we use the balanced queue size property to show in Lemma \ref{l3} that for any arbitrary policy $\pi$ outside the set $\Pi^\text{MWM}$, we may construct a policy in the set $\Pi^\text{MWM}$ that improves $\pi$ in terms of delay. In the words of Theorem \ref{th1}, we prove that the delay-optimal policy belongs to the set of MWM policies $\Pi^\text{MWM}$.

Third, in Lemmas \ref{lnew} and \ref{lperm} we will show that all policies in the set $\Pi^\text{MWM}$ result in the same cost (and hence average delay).
Finally, by using Theorem \ref{th1} and Lemma \ref{lnew} we conclude Theorem \ref{th2} where we show that the policies in the set $\Pi^\text{MWM}$ are all delay-optimal.
Graph theoretic analysis is applied in the proof of Lemmas \ref{l2} and \ref{lperm} and stochastic ordering and dynamic coupling arguments are used to prove Lemmas \ref{l3} and \ref{lnew}.

%
%

\vspace{-3mm}
\subsection{Equivalence of Queue Length Balancing and Maximum Weighted Matching}

We start this section by introducing the \textit{intermediate} queue state in the following definition.
\begin{defin}
Let $\mat{X}'(t)=(X'_1(t),X'_2(t),\dots,X'_N(t))$ denote the queue length vector at time slot $t$ exactly \textit{after serving the queues according to a server assignment policy $\pi$ and before adding the new arrivals} of time slot $t$, i.e.,
\begin{eqnarray} \label{xprim}
X'_n(t)= \left (X_n(t-1)-\displaystyle\sum_{k=1}^{K}C_{n,k}(t)M_{n,k}^{(\pi)}(t) \right)^+.
\end{eqnarray}
We call this vector as a the \textit{intermediate} queue state. Recall that the final state of queue $n$ at time slot $t$ is determined after adding the new arrivals.
\end{defin}

Given $\mat{x}'(t)$ as a sample value of random vector $\mat{X}'(t)$, we define a \textit{balancing server reallocation} at time slot $t$ as follows.
\begin{defin}\label{balancing}
Assume that the employed matching at time slot $t$ (assignment of servers to the queues at time slot $t$) will result in the intermediate queue vector $\mat{x}'(t)$. A balancing server reallocation
at this time slot is \textit{a new matching} resulting in intermediate vector $\tilde{\mat{x}}'(t)$ such that one of the following conditions is satisfied.
\begin{enumerate}[]
\item ($\textbf{C1}$) $\tilde{x}'_n(t) \leq x'_n(t)$ for all $n=1,2,\dots,N$ and there exists an $m \in \{1,2,\dots,N\}$ such that $\tilde{x}'_m(t) < x'_m(t)$.
\item($\textbf{C2}$) $\tilde{\mat{x}}'(t)$ and $\mat{x}'(t)$ are different in only two elements $n$ and $m$ such that $x'_n(t) <\tilde{x}'_n(t)\leq \tilde{x}'_m(t)< x'_m(t)$ and the following constraints are satisfied: $\tilde{x}'_n(t)=x'_n(t)+1$ and $\tilde{x}'_m(t)=x'_m(t)-1$.
\end{enumerate}
\end{defin}
A balancing server reallocation is a crucial tool in defining new policies that improve the delay performance of an arbitrary policy as we will see in the proof later.

\el Consider a system with three queues and three servers. Assume that $\mat{x}(t-1)=(3,2,5)$ is the queue length vector right at the end of time slot $t-1$ (or at the beginning of time slot $t$). We consider two distinct examples to show the definition of balancing server reallocations corresponding to each of the cases $\textbf{C1}$ and $\textbf{C2}$ in Definition \ref{balancing}. 
Figures \ref{examples:1} and \ref{examples:2} show these examples of balancing server reallocations. In each case, we also show the weight of each edge $(n,k)$ which is equal to $c_{n,k}(t)x_n(t-1)$. In these figures, since none of the queues is empty, the edges with weight $0$ are the ones which are disconnected. We have specified the original allocations by solid lines and the balancing ones by dashed lines.
For the  system in Figure \ref{examples:1}, the original allocation will result in the intermediate vector $\mat{x}'(t)=(3,1,4)$ while the balancing server reallocation will result in the intermediate vector $\tilde{\mat{x}}'(t)=(2,1,4)$. The vectors $\mat{x}'(t)$ and $\tilde{\mat{x}}'(t)$  satisfy Condition $\textbf{C1}$.
For the system in Figure \ref{examples:2}, the original allocation will result in the intermediate vector $\mat{x}'(t)=(2,1,5)$ while the balancing server reallocation will result in $\tilde{\mat{x}}'(t)=(3,1,4)$. The vectors $\mat{x}'(t)$ and $\tilde{\mat{x}}'(t)$  satisfy Condition $\textbf{C2}$.

\begin{figure}
  \centering
    \psfrag{a}[][][.85]{$x_1(t-1)=3$}
    \psfrag{b}[][][.85]{$x_2(t-1)=2$}
    \psfrag{c}[][][.85]{$x_3(t-1)=5$}
    \psfrag{d}[][][.85]{$1$}
    \psfrag{e}[][][.85]{$2$}
    \psfrag{f}[][][.85]{$3$}  
    \psfrag{g}[][][.85]{$0$}
    \psfrag{h}[][][.85]{$3$}
    \psfrag{i}[][][.85]{$2$}
    \psfrag{j}[][][.85]{$2$}
    \psfrag{k}[][][.85]{$5$}
    \psfrag{l}[][][.85]{$5$}   
    \psfrag{m}[][][.85]{$x_1(t-1)=3$}
    \psfrag{n}[][][.85]{$x_2(t-1)=2$}
    \psfrag{o}[][][.85]{$x_3(t-1)=5$}
    \psfrag{p}[][][.85]{$3$}
    \psfrag{q}[][][.85]{$0$}
    \psfrag{r}[][][.85]{$2$}  
    \psfrag{s}[][][.85]{$2$}
    \psfrag{t}[][][.85]{$5$}
    \psfrag{u}[][][.85]{$0$}
    \psfrag{v}[][][.85]{$1$}
    \psfrag{w}[][][.85]{$2$}
    \psfrag{x}[][][.85]{$3$} 
  \subfloat[Satisfying condition \textbf{C1}]{\label{examples:1}\includegraphics[width=0.35\textwidth]{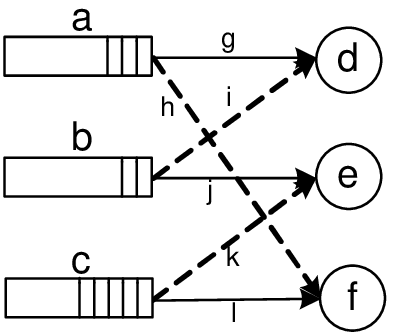}}              
  \subfloat[Satisfying condition \textbf{C2}]{\label{examples:2}\includegraphics[width=0.35\textwidth]{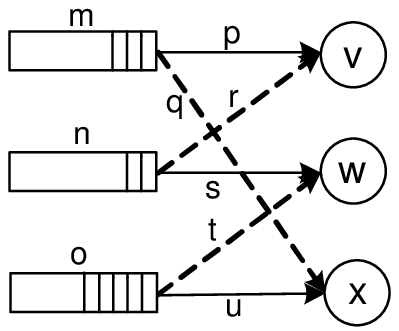}}
  \caption{Examples of balancing server reallocations (the weight $c_{n,k}(t)x_n(t-1)$ of each edge $(n,k)$ is also shown)}
  \label{examples}\vspace{-5mm}
\end{figure}


\begin{defin}\label{mindex}
For a server assignment policy $\pi$ with the allocation variables $\{M_{n,k}^{(\pi)}(t)\}_{t=1}^{\infty}$, $\forall k\in \cal K$ and $\forall n \in \cal N$, 
we define Matching Weight ($\mathsf{MW}$) index at time slot $t$ by 
\begin{eqnarray} 
\mathsf{MW}_{\pi}(t)=\displaystyle\sum_{n=1}^{N} X_n(t-1) \displaystyle\sum_{k=1}^{K}C_{n,k}(t) M_{n,k}^{(\pi)}(t). 
\end{eqnarray}
\end{defin}

$\mathsf{MW}$ index is exactly the objective of the optimization problem (\ref{matching}). $\mathsf{MW}_{\pi}(t)$ is an index associated with policy $\pi$ at time slot $t$ whose value is dependent on the state of the system (queue lengths and connectivities)  as well as the matching employed by policy $\pi$ at time slot $t$. 
In the following lemmas (Lemmas \ref{l1} and \ref{l2}), we relate the notions of balancing server reallocation and Matching Weight index and we prove that maximization of $\mathsf{MW}_{\pi}(t)$ index and balancing of the queues are equivalent. More specifically, we show that if the $\mathsf{MW}_{\pi}(t)$ index for policy $\pi$ is not maximized at time slot $t$ ($\pi$ is not using a maximum weighted matching), then there exists a balancing server reallocation (i.e., a new matching that satisfies either \textbf{C1} or 
\textbf{C2}) that results in a larger $\mathsf{MW}$ index. Furthermore, if $\pi$ is using a maximum weighted matching, then there exists no balancing server reallocation at that time slot, i.e., no matching can be found that satisfies either \textbf{C1} or \textbf{C2}. These facts are formally stated in the following two lemmas.

\begin{lem}\label{l1}
For a given policy $\pi$ employing matching $\mat{M}^{(\pi)}(t)$ at time slot $t$, by applying a balancing server reallocation at time slot $t$ (if there exists any), we can create a new policy $\tilde{\pi}$ (differing from $\pi$ only at time slot $t$) such that $\mathsf{MW}_{\pi}(t) < \mathsf{MW}_{\tilde{\pi}}(t)$.
\end{lem}
The detailed proof of the lemma is given in Appendix \ref{l12}.
Based on Lemma \ref{l1}, we can state the following corollary.
\cor \label{cor1}
For a given policy $\pi$ at time slot $t$, if $\mathsf{MW}_{\pi}(t)$ is maximized, i.e., policy $\pi$ employs a maximum weighted matching at time slot $t$, then there exists no balancing server reallocation at that time slot.

Lemma \ref{l1} states that any balancing server reallocation \textit{strictly} increases the matching weight index. However, \textit{it does not imply the existence} of a balancing server reallocation when $\mathsf{MW}_{\pi}(t)$ is not maximized. In the following, we  prove the existence result i.e., the inverse of Lemma \ref{l1}.
\begin{lem}\label{l2}
For a given policy $\pi$ at time slot $t$, if $\mathsf{MW}_{\pi}(t)$ is not maximized, i.e., if $\mathsf{MW}_{\pi}(t)<\mathsf{MW}_{\text{MWM}}(t)$, then there exists a balancing server reallocation at that time slot.
\end{lem}

For the detailed proof, please refer to Appendix \ref{l13}.
 
Using Lemmas \ref{l1} and \ref{l2}, we can conclude that maximizing the matching weight is equivalent to balancing the queues (in a sense that there is no further matching that can satisfy \textbf{C1} or \textbf{C2} in Definition \ref{balancing}). Hence, an MWM matching will result in the most balanced intermediate queue state where no balancing server reallocation is possible. This property of an MWM matching will be crucial in the proof of Lemma \ref{l3}.


\vspace{-2mm}
\subsection{Background on Stochastic Ordering and Dynamic Coupling}\label{sto}
In this section, we briefly review the concepts of stochastic ordering (stochastic dominance) and dynamic coupling techniques. These concepts are needed in the proof of delay optimality of MWM policy in the rest of our discussion. The reader is encouraged to consult  \cite{stoyan,ross,Lindvall} for more details about stochastic ordering and dynamic coupling.


\begin{defin}
Consider two real-valued, discrete-time stochastic processes $A=\{A(t)\}_{t=1}^{\infty}$ and $B=\{B(t)\}_{t=1}^{\infty}$ in $\mathbb{R}$. We say $A$ is stochastically smaller than $B$ and we write $A \leq_{st} B$ if $\Pr (A(t)>r) \leq \Pr(B(t)>r)$ \textit{for all} $t=1,2,\dots$ and \textit{all} $r \in \mathbb{R}$ \cite{stoyan,ross}. 
\end{defin}

The following two properties of stochastic ordering are useful: if $A\leq_{st} B$, then
\begin{itemize}
\item[(a)] $E[A(t)] \leq E[B(t)]$
\item[(b)] $f(A)\leq_{st} f(B)$ for all non-decreasing functions $f$.
\end{itemize}
Process $A$ is stochastically smaller than $B$, if there exists a process $\tilde{A}=\{\tilde{A}(t)\}_{t=1}^{\infty}$ defined on the same probability space as $B$, has the same probability distribution as $A$ and satisfies $\tilde{A}(t)\leq B(t)$ almost surely (a.s.) for every $t=1,2,\dots$ \cite{ganti}. The last statement is known as coupling of $A$ and $\tilde{A}$. When applying coupling technique, given the process $A$, we construct a \textit{coupled} process $\tilde{A}$ with the same distribution as $A$ and $\tilde{A}(t)\leq B(t)$ a.s. \textit{for all} $t$. This gives us a tool for comparing the processes $A$ and $B$ stochastically when it is infeasible to derive the distributions of $A$ and $B$ (e.g., in our queueing model when comparing the total occupancy process for different server assignment policies). 
\vspace{-3mm}
\subsection{Delay Optimality of MWM}
In this subsection, we will elaborate on proving the delay optimality of any MWM policy. We first introduce some definitions. We denote by $\mathbb{Z}_+$ the set of non-negative integers and by $\mathbb{Z}_+^N$ the $N$ dimensional Cartesian space of non-negative integers.
We define the relation ``$\preceq$'' on $\mathbb{Z}_+^N$ as follows.

\begin{defin} \label{def_order}
For two vectors $\mat{x}$ , $\tilde{\mat{x}} \in \mathbb{Z}_+^N$, we write $\tilde{\mat{x}} \preceq \mat{x}$ if one of the following relations holds:
\begin{enumerate}[]
\item \textbf{D1}: $\tilde{x}_n \leq x_n$ for all $n=1,2,\dots,N$.
\item \textbf{D2}: $\tilde{\mat{x}}$ is obtained by permutation of two distinct elements of $\mat{x}$, i.e., $\tilde{\mat{x}}$ and $\mat{x}$ are different in only two elements $n$ and $m$ such that $\tilde{x}_{n}=x_m$ and $\tilde{x}_{m}=x_n$. In this case, we say $\tilde{\mat{x}}$ and $\mat{x}$ are \textit{equal in permutation} and we write $\tilde{\mat{x}} \buildrel p\over = \mat{x}$.
\item \textbf{D3}: $\tilde{\mat{x}}$ and $\mat{x}$ are different in only two elements $n$ and $m$ such that $x_n <\tilde{x}_n\leq\tilde{x}_m< x_m$ and the following constraints are satisfied: $\tilde{x}_n=x_n+1$ and $\tilde{x}_m=x_m-1$.
\end{enumerate}
\end{defin}
The three relations \textbf{D1}, \textbf{D2} and \textbf{D3} are mutually exclusive. In \textbf{D3}, we say that $\tilde{\mat{x}}$ is more balanced than $\mat{x}$ and can be obtained by decreasing a larger element of $\mat{x}$ (i.e., $m$) by one and increasing a smaller element (i.e., $n$) by one. We call such an interchange as a \textit{balancing interchange} on vector $\mat{x}$. Thus, the result of a balancing interchange on a vector $\mat{x}$ would be a vector $\tilde{\mat{x}}$ such that $\tilde{\mat{x}} \preceq \mat{x}$. According to Definition \ref{balancing}, a balancing server reallocation satisfying Condition $\textbf{C2}$, will result in a balancing interchange between $\mat{x}'(t)$ and $\tilde{\mat{x}}'(t)$. 

We define the partial order ``$\preceq_p$'' on $\mathbb{Z}_+^N$ as the transitive closure of relation`` $\preceq$'' \cite{lidl}. In other words, $\tilde{\mat{x}} \preceq_p \mat{x}$ if and only if $\tilde{\mat{x}}$ is obtained from $\mat{x}$ by performing a sequence of 
reductions (i.e., reducing an element of the vector $\mat{x}$ such that $\mat{x}$ and $\tilde{\mat{x}}$ satisfy $\textbf{D1}$), permutations of two elements (permutation of two elements of the vector $\mat{x}$ such that $\mat{x}$ and $\tilde{\mat{x}}$ satisfy $\textbf{D2}$) and/or balancing interchanges (such that $\mat{x}$ and $\tilde{\mat{x}}$ satisfy $\textbf{D3}$). When $\mat{x}$ and $\tilde{\mat{x}}$ are two queue length vectors, we write $\tilde{\mat{x}} \preceq_p \mat{x}$ if and only if queue length vector $\tilde{\mat{x}}$ is obtained from $\mat{x}$ by applying a sequence of packet removals, two-queue permutations and balancing interchanges.

\begin{defin}
We define $\cal F$ as the class of real-valued functions on $\mathbb{Z}_+^N$ that are monotone and non-decreasing with respect to the partial order $\preceq_p$, i.e., 
\begin{eqnarray} \label{functions}
f\in \mathcal{F}~~\Longleftrightarrow~~\tilde{\mat{x}}\preceq_p \mat{x} \Rightarrow f(\tilde{\mat{x}}) \leq f(\mat{x}).
\end{eqnarray}
\end{defin}
We can easily check that function $f(\mat{x})= \sum_{n=1}^N x_n $ belongs to $\cal F$. This function represents the total queue occupancy of the system.

\begin{defin}
We define $\Pi_t, t=1,2,\dots,$ as the set of all policies that employ maximum weighted matching in every time slot $\tau=1,\dots,t$.
\end{defin}
 
We observe that $\Pi_{t-1} \supseteq \Pi_{t}$ and $ \Pi^\text{MWM}=\bigcap_{t=1}^{\infty}\Pi_t$.

Consider a policy $\pi \in \Pi_{t-1}$ which is using an arbitrary matching $\mat{M}^{(\pi)}(t)$ at time slot $t$. 
If $\mat{M}^{(\pi)}(t)$ is not a maximum weighted matching,
then from Lemmas \ref{l1} and \ref{l2} we conclude that by applying a sequence of balancing server reallocations\footnote{According to Lemma \ref{l1}, each balancing server reallocation \textit{strictly} increases the matching weight index.} we can create a policy $\pi^{\star} \in \Pi_t$. Let $h_t^{\pi}$ denote the number of balancing server reallocations required to convert the employed matching in policy $\pi$ at time slot $t$ to a maximum weighted matching.  
\begin{defin}
We define the \textit{distance} of policy $\pi \in \Pi_{t-1}$ from the set $\Pi_t$ to be $h_t^{\pi}$ balancing server reallocations.
\end{defin}


According to Lemmas \ref{l1} and \ref{l2}, since by applying each server reallocation, the matching weight index \textit{strictly} increases, the number of balancing server reallocations needed to convert $\pi$ to a maximum weighted matching is bounded, i.e., $h_t^{\pi}\le H <\infty$ for all $t,\pi$.
Hence, after applying the first balancing server reallocation at time slot $t$ we reach a policy $\tilde{\pi}_1$ whose distance from $\Pi_t$ is $h_t^{\pi}-1$ balancing server reallocations. By repeating this procedure we finally identify a policy whose distance to $ \Pi_t $ is zero, i.e., it belongs to $ \Pi_t $. Figure \ref{balancing_reallocation} illustrates the definition of the distance $h_t^{\pi}$ and how balancing server reallocations result in identifying a policy that employs a maximum weighted matching at time slot $t$. In this figure, $\mat{x}'_{\pi}(t)$ is the intermediate queue state due to the employed matching at time slot $t$ and $\mat{x}'_{\tilde{\pi}_1}(t),\mat{x}'_{\tilde{\pi}_2}(t),\dots,\mat{x}'_{\tilde{\pi}_{h_t^{\pi}}}(t)$ are the intermediate queue states after applying the balancing server reallocations.


\begin{defin}
By $\Pi_t^h$ ($0 \leq h \leq H$) we denote the set of all server assignment policies in $ \Pi_{t-1}$ whose distance from $\Pi_t$ is 
$h$ balancing server reallocations. 
\end{defin}
Recall that $\Pi_t^0=\Pi_t$. 


\begin{figure}[tp]
    \centering
    \psfrag{a}[][][.8]{employed matching $\pi$}
    \psfrag{b}[][][.8]{first balancing reallocation $\tilde{\pi}_1$}
    \psfrag{c}[][][.8]{second balancing reallocation $\tilde{\pi}_2$}
    \psfrag{d}[][][.8]{$h_t^{\pi}$th balancing reallocation $\tilde{\pi}_{h_t^{\pi}}$}
    \psfrag{e}[][][1.2]{$h_t^{\pi}$}
    \psfrag{f}[][][.85]{$\mathsf{MW}_{\pi}(t) < \mathsf{MW}_{\text{MWM}}(t)$}
    \psfrag{g}[][][.85]{$\mathsf{MW}_{\pi}(t)<\mathsf{MW}_{\tilde{\pi}_1}(t) < \mathsf{MW}_{\text{MWM}}(t)$}
    \psfrag{h}[][][.85]{$\mathsf{MW}_{\tilde{\pi}_1}(t)<\mathsf{MW}_{\tilde{\pi}_2}(t) < \mathsf{MW}_{\text{MWM}}(t)$}
    \psfrag{i}[][][.85]{$\mathsf{MW}_{\tilde{\pi}_{h_t^{\pi}}}(t) = \mathsf{MW}_{\text{MWM}}(t)$}
    \psfrag{W}[][][1]{$\mat{x}(t)$}
    \psfrag{t}[][][1]{time slot $t$}
    \psfrag{X}[][][1]{$\mat{x}(t-1)$}
    \psfrag{j}[][][1]{$\mat{x}'_{\pi}(t)~~\Longrightarrow$}
    \psfrag{k}[][][1]{$\mat{x}'_{\tilde{\pi}_1}(t)~\Longrightarrow$}
    \psfrag{l}[][][1]{$\mat{x}'_{\tilde{\pi}_2}(t)~\Longrightarrow$}
    \psfrag{m}[][][1]{$\mat{x}'_{\tilde{\pi}_{h_t^{\pi}}}(t)~\Longrightarrow$}
    \psfrag{v}[][][.7]{This balancing reallocation creates maximum weight matching}
    \includegraphics[width=.8\textwidth]{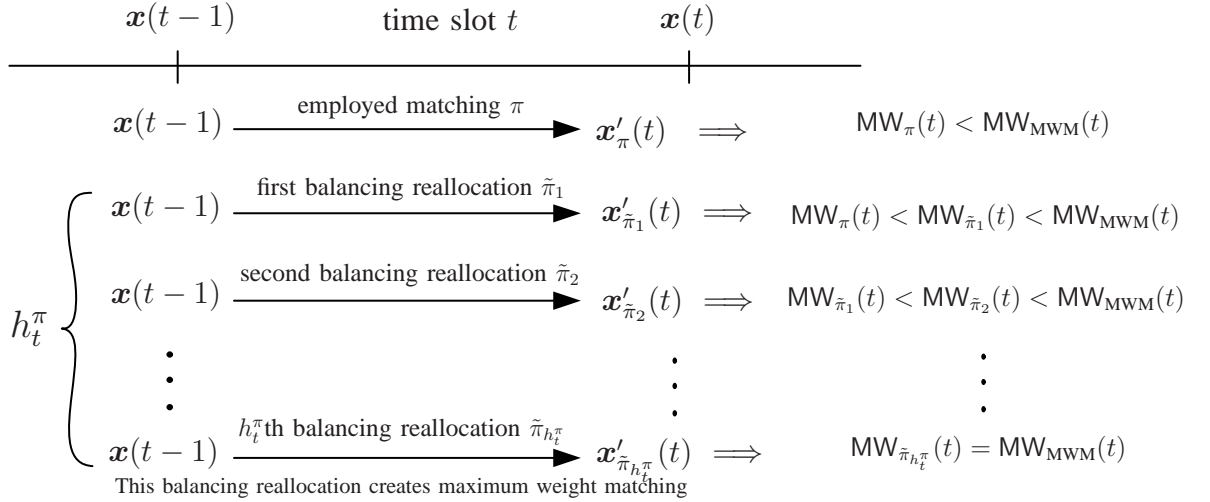}
    \caption{$h_t^{\pi}$ balancing server reallocations are required to create a policy in $\Pi_t$ from policy $\pi \in \Pi_{t-1}$. $\mathsf{MW}$ indices given the state of the system at time slot $t$ are also compared.}
    \label{balancing_reallocation}\vspace{-5mm}
\end{figure}

\begin{defin}
For any two policies $ \pi$ and $\tilde{\pi}$ with queue length processes $\mat{X}=\{\mat{X}(t)\}_{t=1}^{\infty}$ and $\tilde{\mat{X}}=\{\tilde{\mat{X}}(t)\}_{t=1}^{\infty}$, respectively, we say $\tilde{\pi}$ \textit{dominates} $\pi$, if $f(\tilde{\mat{X}}) \leq_{st} f(\mat{X})$, $f\in\cal F$, i.e., the queue length cost (delay) of policy $\tilde{\pi}$ is stochastically less than that of policy $\pi$.
\end{defin}

If $\tilde{\pi}$ dominates $\pi$ we have $ E[f(\tilde{\mat{X}})] \leq E[f(\mat{X})] $\footnote{Choosing $f(\mat{x})= \sum_{n=1}^N x_n $, we conclude that the expected total queue occupancy (or equivalently average queueing delay) of policy $\tilde{\pi}$ is smaller than that of policy $\pi$ in every time slot.}.  
In the following lemma, we will interconnect the notions of ``maximizing the matching weight index'' and ``delay optimality'' and show that maximization of the matching weight index (at any given time $t$) will improve the delay performance (will decrease the queue length cost function $f(\mat{X})$ stochastically). The key element in the interconnection is the notion of balancing server reallocation. In particular, we show that, for any given policy $ \pi \in \Pi_t^h$, $h=h_t^{\pi}$  that does not employ a maximum weighted matching at time slot $t$ (i.e., $h>0$), there exists a balancing server reallocation at time slot $t$. In the following lemma, we show that by using such a balancing server reallocation at time slot $t$ we can construct a new policy $\tilde{\pi}$ that dominates the original policy $\pi$. For the detailed proof, please refer to Appendix \ref{l14}. We used stochastic ordering and dynamic coupling to prove this lemma.

\begin{lem} \label{l3}
For any policy $ \pi \in \Pi_t^h$ 
where $h=h_t^{\pi}>0$, we can construct a policy $\tilde{\pi} \in \Pi_t^{h-1}$ 
such that $\tilde{\pi}$ dominates $\pi$. Thus, $\tilde{\pi}$ outperforms $\pi$ in terms of average queueing delay. 
\end{lem}

Using Lemma \ref{l3}, we can prove the following theorem which states that any MWM policy outperforms any non-MWM policy in terms of average queueing delay.

\begin{theor} \label{th1}
For any server assignment policy $\pi \notin \Pi^{\text{MWM}}$, there exists an MWM policy $\pi^\ast\in \Pi^{\text{MWM}}$ such that $\pi^\ast$ dominates $\pi$.
\end{theor}

\begin{proof}
Let $\pi$ be any arbitrary non-MWM policy. Then $\pi \in \Pi_1^{H_1}$ where $H_1={h^{\pi}_1}$. By applying Lemma \ref{l3} repeatedly, we can construct a sequence of policies such that each policy dominates the previous one. Thus, we obtain policies that belong to $\Pi_1^{H_1},\Pi_1^{{H_1}-1},\Pi_1^{{H_1}-2},\dots,\Pi_1^0=\Pi_1$. The last policy is called $\pi_1$ for which we have $\pi_1 \in \Pi_2^{H_2}$ where $H_2=h^{\pi_1}_2$. By continuing such an argument, we obtain a sequence of policies $\pi_t \in \Pi_t$, $t=1,2,\dots$ such that $\pi_j $ dominates $\pi_i$ for $j>i$.
This sequence of policies defines a limiting policy $\pi^\ast$ that agrees
with MWM at all time slots. Thus, $\pi^\ast $ is an MWM policy that dominates all the
previous policies, including the starting policy $\pi$. This proves that the delay-optimal policy is an MWM policy in $\Pi^{\text{MWM}}$.
\end{proof}

As we mentioned before, the set $\Pi^{\text{MWM}}$ may contain an infinite number of policies. 
In the following, we show that \textit{any} MWM policy is delay-optimal.
To achieve this, we need to prove the following lemma\footnote{As part of the proof for this lemma, we need preliminary Lemma \ref{lperm} presented and proven in Appendix \ref{plperm}}. 




\begin{lem} \label{lnew}
The queue length costs of all the maximum weighted matching policies in $\Pi^\text{MWM}$ are equal in distribution, i.e., for any two MWM policies $\pi_1,\pi_2 \in \Pi^{\text{MWM}}$, we have $f({\mat{X}}^{(\pi_1)}) \buildrel \mathcal{D}\over = f({\mat{X}}^{(\pi_2)})$ where ${\mat{X}}^{(\pi_1)}$ and ${\mat{X}}^{(\pi_2)}$ are the queue length processes under $\pi_1$ and $\pi_2$, respectively.
\end{lem}
The proof of this lemma is provided in Appendix \ref{plnew}. 


Using Theorem \ref{th1} and Lemma \ref{lnew}, we can conclude the main result of this section in the following theorem.
\begin{theor} \label{th2}
Any Maximum Weighted Matching policy dominates any server assignment policy, i.e., any MWM policy is delay-optimal.
\end{theor}





\vspace{-3mm}
\subsection{Extensions}
\subsubsection{Imperfect Services}
We can extend Theorems \ref{th1} and \ref{th2} for the case where the service of a scheduled packet by a connected server fails randomly with a certain probability. This can model the operation of realistic wireless networks where service failures usually occur due to unexpected and unpredictable effects of noise, interference, etc. In the case of a packet service failure, the packet will be kept in the queue and will be rescheduled and retransmitted in future time slots.

By the random variable $Q_{n,k}(t) \in \{0,1\}$, we denote the successful/unsuccessful service of queue $n$ provided by server $k$ at time slot $t$; a value of $1$ (resp. $0$) denotes that the service is successful (resp. unsuccessful). We assume that $Q_{n,k}(t)$, $\forall n \in \mathcal{N}, \forall k \in \mathcal{K}$ are i.i.d. Bernoulli random variables with the same success probability $q$. The parameter $q$ (similar to parameters $\lambda$ and $p$) is not explicitly involved in our analysis other than the fact that 
 $E[Q_{n,k}(t)]=q, \forall n,k,t$. The queue lengths are then updated at the end of each time slot by the following rule. 
\vspace{-1mm}
\begin{eqnarray}
X_n(t)= \left (X_n(t-1) -\displaystyle\sum_{k=1}^{K}C_{n,k}(t)M_{n,k}^{(\pi)}(t) Q_{n,k}(t) \right)^+ + A_n(t)~~~~~\forall n \in \cal N
\end{eqnarray}
The network scheduler (that performs server assignment process) cannot observe the variables $Q_{n,k}(t)$ and from its perspective they are assumed to be random. The random vector $\mat{X}'(t)$ is defined similar to equation (\ref{xprim}). Hence, $\mat{X}'(t)$ represents the queue lengths before adding the new arrivals of time slot $t$ as if all the services at that time slot are successful. 
   
For such a system, we can verify that Lemmas \ref{l1} and \ref{l2} are valid. We can extend Lemma \ref{l3} for the system with service failures by considering the random variables $Q_{n,k}(t)$ in our dynamic coupling argument. 
The proof is followed by using the same approach as in Lemma \ref{l3}. The detailed analysis is brought in Appendix \ref{wsf}.
%
%
By applying the same approach as in the proof of Theorem \ref{th1} and Lemma \ref{lnew}, we can similarly prove the delay optimality of MWM policy for the system with imperfect services.

\subsubsection{Extensions for Connectivity and Arrival Processes}
The arguments in Lemmas \ref{l3} and \ref{lnew}  and Theorem \ref{th1} remain valid if the i.i.d. assumption for connectivity and arrival processes is relaxed as follows; we will consider connectivity and arrival processes which follow conditional permutation invariant distributions. Given event $\cal H$ (which is used to denote the history of the system), we define a conditional multivariate probability distribution $f(y_1,y_2,\dots,y_n \mid \mathcal{H})$ to be permutation invariant if for any permutation of the variables $y_1,y_2,\dots,y_n$ namely $y'_1,y'_2,\dots,y'_n$ we have $f(y_1,y_2,\dots,y_n \mid \mathcal{H})=f(y'_1,y'_2,\dots,y'_n \mid \mathcal{H})$.
We can readily see that for all the connectivity and arrival processes whose joint distributions at each time slot given the history of the system\footnote{By history of the system we mean all the channel states, arrivals and matchings of the previous time slots up to time slot $t$.} (i.e., $f_{\mat{A}(t)}(a_1,a_2,..,a_N\mid \mathcal{H})$ and $f_{\mat{C}(t)}(c_{1,1},c_{1,2},\dots,c_{N,K-1},c_{N,K}\mid \mathcal{H})$) are permutation invariant, Lemmas \ref{l3}, \ref{lnew} and Theorem \ref{th1} are still valid and therefore MWM is delay-optimal.

We also consider the generalization of Theorems \ref{th1} and \ref{th2} for non-Bernoulli arrival processes. 
Suppose that the number of arrivals to each queue can be represented by the summation of some i.i.d. Bernoulli random variables, i.e., has Binomial distribution. Also suppose that $A_n(t)\leq A_{\text{max}}$ \textit{for all} $n \in \cal N$ and \textit{all} $t$. In this case, we can create a new (virtual) system in which after each time slot we append $A_{\text{max}}-1$ virtual time slots and put the connectivities all equal to zero, i.e., for each virtual time slot $t$, $C_{n,k}(t)=0, \forall n \in \mathcal{N}, \forall k \in \mathcal{K}$. We then distribute the arrivals of the actual time slot among these $A_{\text{max}}$ time slots (one actual time slot and $A_{\text{max}}-1$ virtual time slots) randomly such that at each time slot at most one packet arrival occurs. Since the connectivities and the arrivals in both systems are \textit{permutation invariant}, we can still prove Theorems \ref{th1} and \ref{th2} for the virtual system. We observe that the operation of the two systems (the original system and the virtual system) are the same. Therefore, we can conclude that Theorem \ref{th1} is also valid for a multi-server system with Binomial arrival processes. 
\vspace{-2mm}
\section{Simulation Results} \label{sim}
We have compared the delay performance of MWM policy with two alternative server assignment policies described in the following.


\begin{itemize}
\item \textit{Maximum Matching} (MM) policy applies the maximum matching on matrix $\mat{C}(t)$. The maximum matching policy at each time slot $t$ employs a server assignment (or matching) $\mat{M}^{(\text{MM})}(t)$ which is obtained by solving the following problem (equivalent to finding the maximum matching in the connectivity matrix). 
\begin{eqnarray}\label{matching2}
\texttt{Maximize:}&~ \displaystyle\sum_{n=1}^{N} \displaystyle\sum_{k=1}^{K} M_{n,k}(t)C_{n,k}(t)~~~~~~~~~~~ \nonumber \\
\texttt{Subject to:}&~~~~~~~\displaystyle\sum_{k=1}^{K} M_{n,k}(t)\leq 1, ~~ (n=1,2,\dots,N), \nonumber\\
&~~~~~~~\displaystyle\sum_{n=1}^{N} M_{n,k}(t)\leq 1, ~~ (k=1,2,\dots,K).
\end{eqnarray} 

The MM maximizes the instantaneous throughput at each time slot without considering the queue length information in its server assignment decisions.

\item A heuristic policy that assigns the servers to the queues at each time slot according to the following rule: It selects a server randomly and assigns it to its longest connected queue. Then, updates the set of servers by removing the selected server from $\cal K$ and the set of queues (i.e., $\mathcal{N}$) by removing the queue to which the selected server was assigned. This procedure is repeated $K$ times. For some servers the updated set $\cal N$ may be empty (e.g., when $K>N$) and therefore those servers are not assigned to any queue.
\begin{algorithm}[H]
\caption{Heuristic Policy Pseudocode}
\noindent
{\bf input}: $\cal N$, $\cal K$, $\mat{c}(t)$ and $\mat{x}(t-1)$ \\
{\bf initialize}: $\mat{M}^{(H)}(t)=(\mat{0})_{N\times K}$ \\
{\bf for} $i=1$ {\bf to} $K$ {\bf do}\\
\indent\indent~~~~~ Choose a server $k^{\star} \in \cal K$ randomly \\
\indent~~~~~~{\bf if} $\mathcal{N}\neq \emptyset$\\
\indent\indent~~~~~~~~~~~ $n^{\star}\longleftarrow\arg\max_{n\in \mathcal{N}} c_{n,k^{\star}}(t)x_{n}(t-1)$ \\
\indent\indent~~~~~~~~~~~ $M_{n^{\star},k^{\star}}^{(H)}(t)\longleftarrow1$ \\
\indent\indent~~~~~~~~~~~ $\mathcal{N}\longleftarrow\mathcal{N}-\{n^{\star}\}$\\
\indent~~~~~~{\bf endif} \\
\indent~~~~~ $\mathcal{K}\longleftarrow\mathcal{K}-\{k^{\star}\}$\\
{\bf end}\\
{\bf Return $\mat{M}^{(H)}(t)$}
\end{algorithm}
\end{itemize}
The motivation for the heuristic policy is coming from Longest Connected Queue (LCQ) policy which was proven in \cite{Tassiulas93} to be optimal for a single-server system. For multi-server system, we will use the same principle for each server. However, the order in which servers are selected for assignment is random.


We have preformed a comprehensive set of simulations in which we investigate the effects of the number of servers $K$, the probability of connectivity $p$ and the probability of service success $q$ on the performance of the aforementioned policies. In all the simulations, we set $N=8$ and the arrivals are i.i.d. Binomial distributed which is the summation of 10 Bernoulli random variables. We use log-scale for the y-axis in the figures so that we can easily compare the performance of different policies in low arrival rates (where the average queue lengths are very close for different policies).
\begin{figure}[tp]
\begin{center}
\subfloat[$p=0.2$]{\includegraphics[width=.7\textwidth]{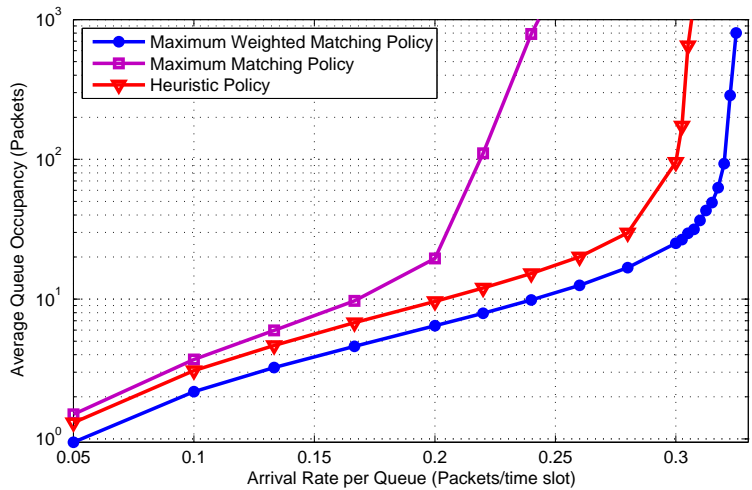}} \\
\subfloat[$p=0.5$]{\includegraphics[width=.7\textwidth]{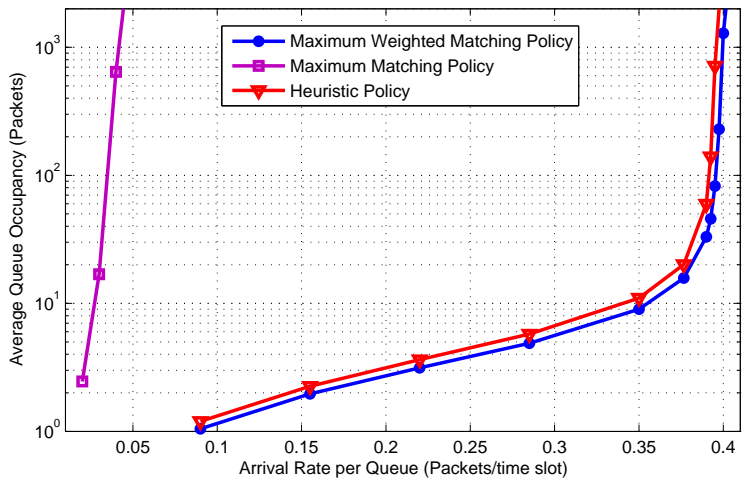}} 
\vspace{-1mm}\caption{Average total queue occupancy, $N=8$, $K=4$, $q=0.8$}\vspace{-8mm}  \label{4}
\end{center}
\end{figure}
\begin{figure}[tp]
\begin{center}
\subfloat[$p=0.2$]{\includegraphics[width=.7\textwidth]{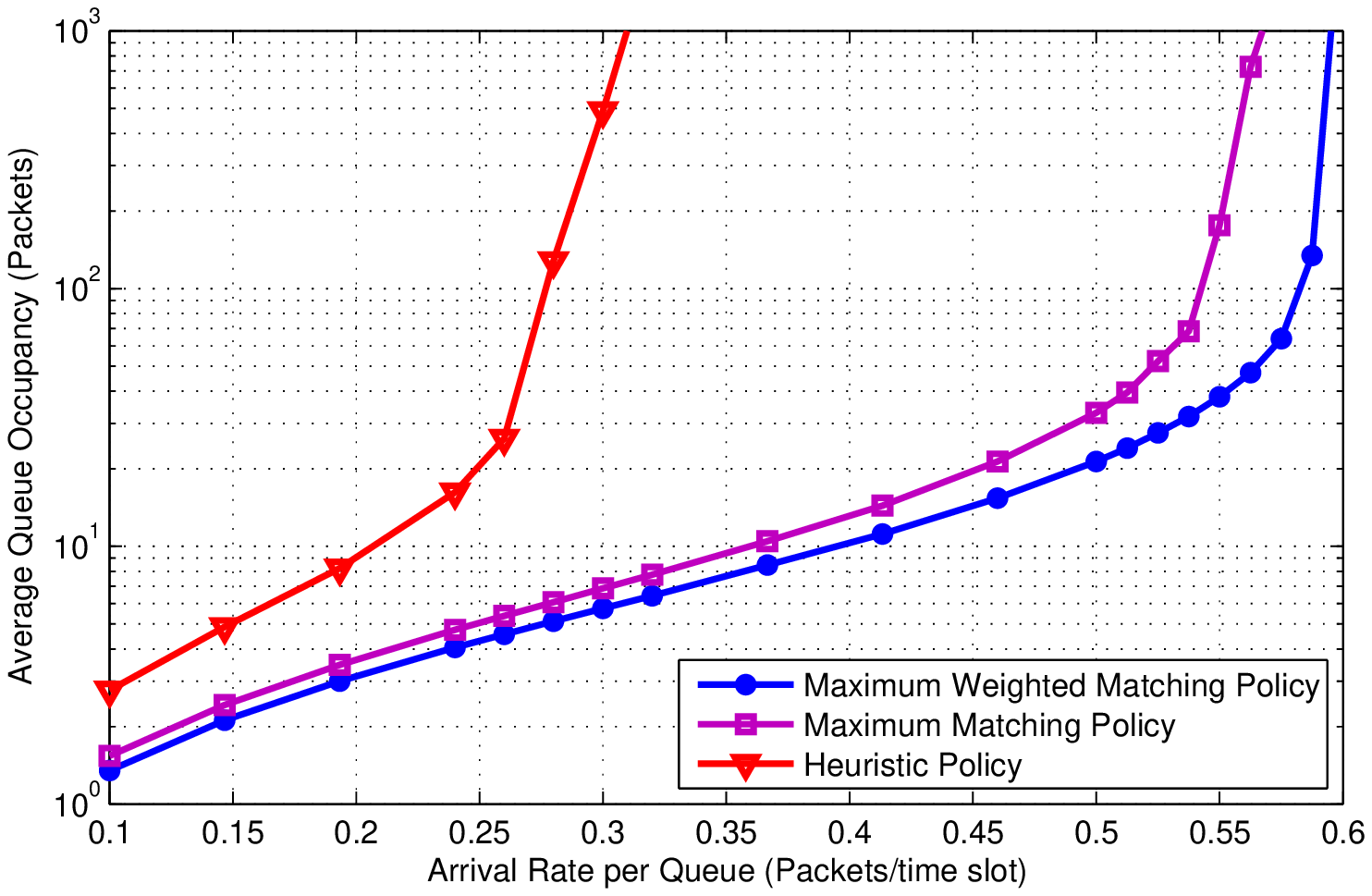}} \\
\subfloat[$p=0.5$]{\includegraphics[width=.7\textwidth]{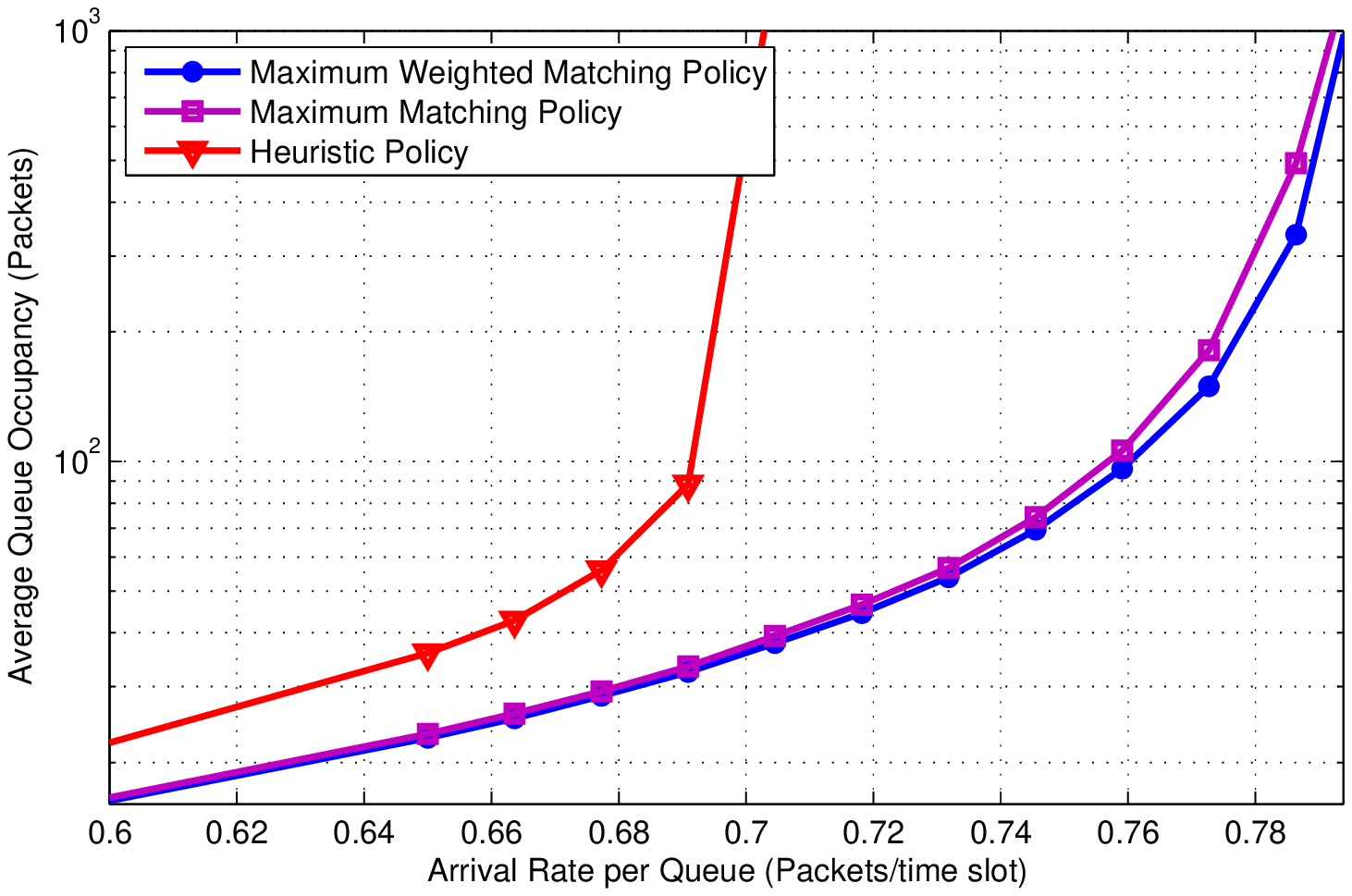}}
\vspace{-1mm}\caption{Average total queue occupancy, $N=8$, $K=8$, $q=0.8$}\vspace{-8mm} \label{8}
\end{center}
\end{figure}
\begin{figure}[tp]
\begin{center}
\subfloat[$q=0.8$]{\includegraphics[width=.7\textwidth]{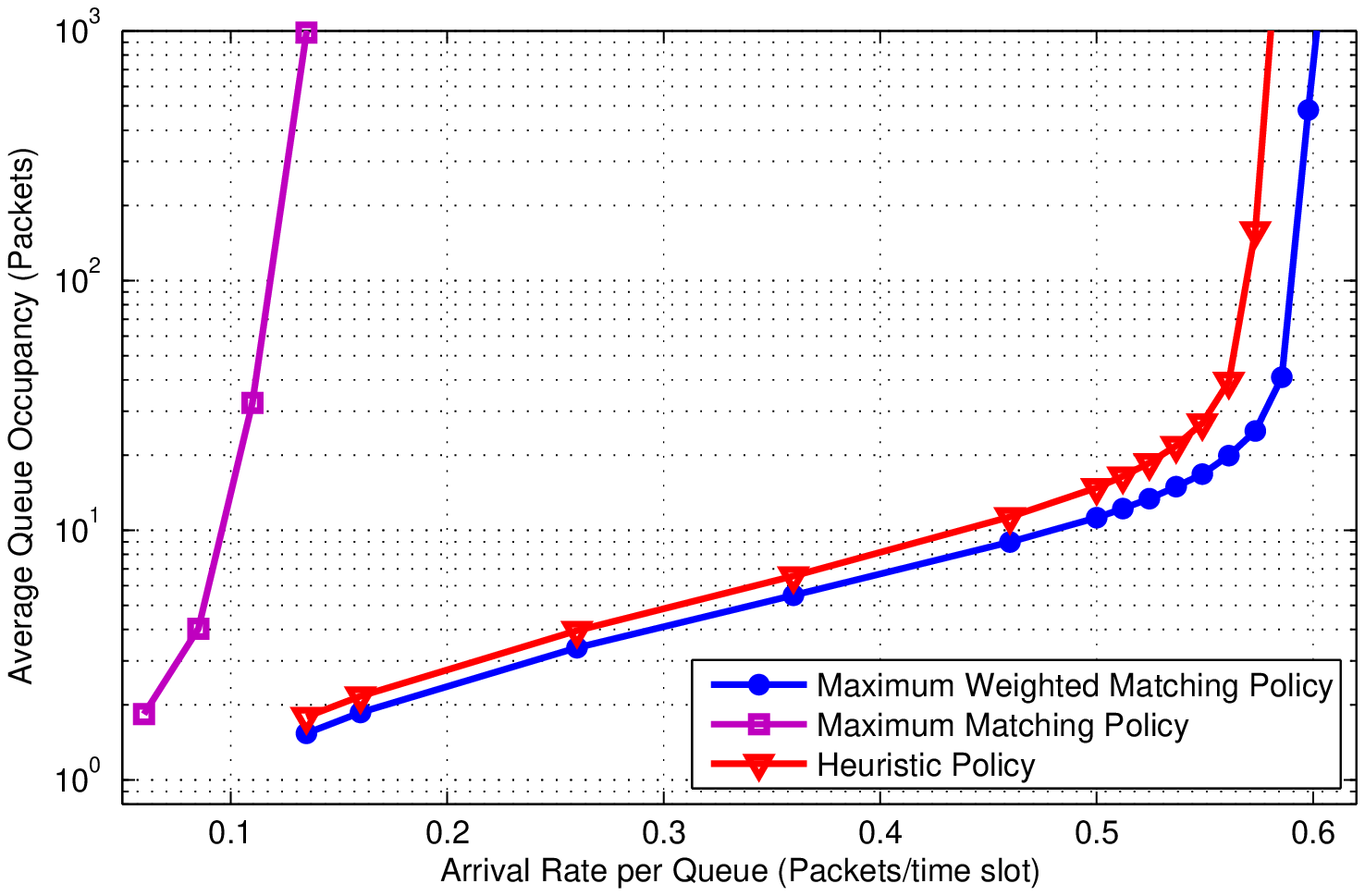}} \\
\subfloat[$q=0.2$]{\includegraphics[width=.7\textwidth]{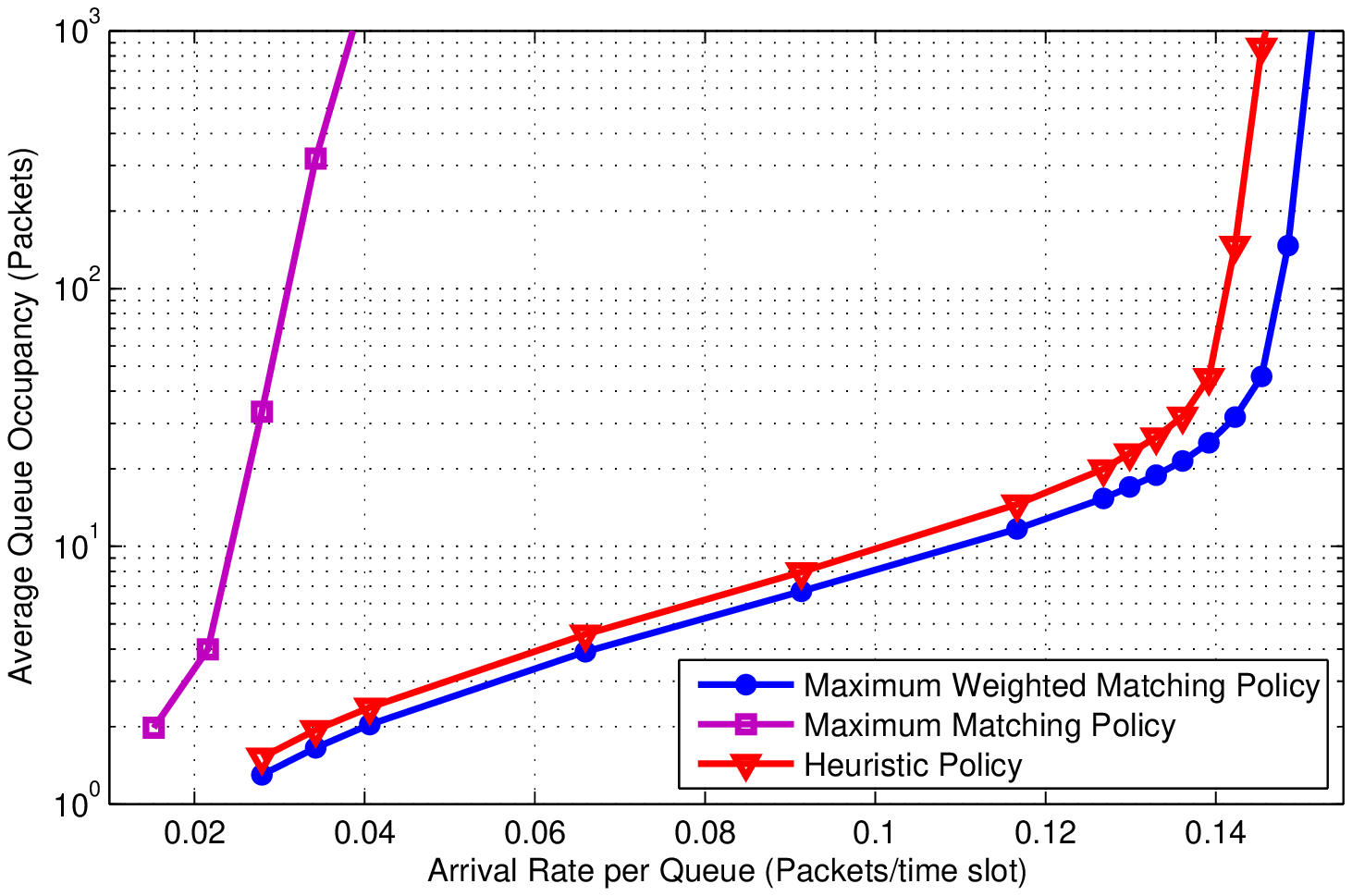}}
\vspace{-1mm}\caption{Average total queue occupancy, $N=8$, $K=6$, $p=0.5$}\vspace{-7mm} \label{6}
\end{center}
\end{figure}

Figures \ref{4}-\ref{6} illustrate the simulation results. In all the cases, the confidence interval is very small and is not visible in the graphs. 
 As we can see in all cases, MWM exhibits improved performance with respect to the other policies in terms of average queue occupancy or average queueing delay.
Figure \ref{4} shows the simulation results for $K=4$, $q=0.8$ and $p=0.2$, $0.5$. In these cases, since the number of servers is relatively low, server assignment will be more competitive. 
As shown earlier, the MWM minimizes the queue imbalance.
The heuristic policy follows the same principle. However, since the selection of servers for assignment is random, in certain cases it may happen that 
two or more servers have the same longest connected queue. In such cases, the order of selecting the servers for assignment does have an effect on the system performance. Maximum Matching policy however, does not try to balance the queues since by construction it does not consider the queue lengths in its assignments and that is why it performs worse than the other two policies. We observe that as the connectivity probability gets larger, the performances of MWM and the heuristic policy get closer. 
%
It is worth mentioning that the heuristic policy introduced here performs the same as MWM for $K=1$ (which is equivalent to LCQ whose optimality has been previously shown in \cite{Tassiulas93}).

Figure \ref{8} shows the results for $8$ servers. In this case, since the number of servers is relatively large and comparable to the number of queues, in MWM and MM policies each queue gets service with high probability when the probability of server connectivities increases. As the connectivity probability gets smaller, the difference in performance of MWM and MM becomes more apparent. In this case, the heuristic policy performs worse than the other two policies since it is more probable to lead to cases where two or more servers have the same longest connected queue. As the number of servers increases, we expect MM to perform the same as MWM as in this case the probability of serving all the queues increases. Therefore, in the limiting case where $K$ becomes very large, MM and MWM result in very close performance.

In Figure \ref{6} we have investigated the effect of service success probability. As we can see in the figures, the only effect of this parameter is to change the stability point (the arrival rate at which queue occupancy tends to infinity). In this case, again we can see that for both $q=0.2$, $0.8$, MWM policy outperforms the other policies.

\section{Conclusions}\label{conc}
In this paper, we considered the problem of assigning $K$ identical servers to a set of $N$ parallel queues in a time-slotted, multi-server queueing system with random connectivities. For such systems, it has been previously shown that MWM is throughput-optimal, i.e., has the maximum stability region. In this paper, we showed that for a system with i.i.d. Bernoulli arrival and connectivity processes, MWM is also optimal for minimizing a class of cost functions of queue lengths including the average queueing delay. 
We first proved that MWM and queue length balancing are equivalent. Then, using this result and by applying the notions of stochastic ordering and dynamic coupling techniques, we proved the delay optimality of MWM. 
Finally, we considered extensions of the model in which we have imperfect packet services or more general packet arrival and server connectivity processes. We have shown the optimality of MWM in these cases as well.

\appendices

\section{Proof of Lemma \ref{l1} and Lemma \ref{l2}} 
\subsection{Proof of Lemma \ref{l1}}\label{l12}
\begin{proof}
Let $\mat{M}^{(\tilde{\pi})}(t)$ denote the employed matching after applying the balancing server reallocation. According to the definition of balancing server reallocation, a server reallocation at time slot $t$ results in an intermediate queue length vector $\tilde{\mat{x}}'(t)$ that satisfies either condition \textbf{C1} or \textbf{C2}. Therefore, we consider the following two cases:

\textit{Case 1}: Condition \textbf{C1} is satisfied at time slot $t$. Thus, $\tilde{x}'_i(t) \leq x'_i(t)$ for all $i=1,2,\dots,N$ and there exists \textit{at least} one $m \in \{1,2,\dots,N\}$, such that $0\leq \tilde{x}'_m(t) < x'_m(t)$. We denote the (sub)set of queues for which we have $0\leq \tilde{x}'_i(t) < x'_i(t)$ by $\cal Q$.
Therefore, there exists no queue that was served by  policy $\pi$ but not by policy $\tilde{\pi}$. Also the queues in subset $\cal Q$ which were not receiving service by policy $\pi$ at time slot $t$, are now receiving service after applying the balancing server reallocation. Therefore, for all $i\in \mathcal{Q}$, $x_i(t-1) \sum_{k=1}^{K}c_{i,k}(t) M_{i,k}^{(\pi)}(t)=0 $, $x_i(t-1)\sum_{k=1}^{K}c_{i,k}(t) M_{i,k}^{(\tilde{\pi})}(t)=x_i(t-1)>0 $. Thus,
\begin{eqnarray}
\lefteqn{\mathsf{MW}_{\pi}(t)= \sum_{i\notin \mathcal{Q}} x_i(t-1) \sum_{k=1}^{K}c_{i,k}(t) M_{i,k}^{(\pi)}(t)+\sum_{i \in \mathcal{Q}} x_i(t-1) \sum_{k=1}^{K}c_{i,k}(t) M_{i,k}^{(\pi)}(t)} \nonumber \\
&&
<  \sum_{i\notin \mathcal{Q}} x_i(t-1) \sum_{k=1}^{K}c_{i,k}(t) M_{i,k}^{(\tilde{\pi})}(t)+\sum_{i \in \mathcal{Q}} x_i(t-1) \sum_{k=1}^{K}c_{i,k}(t) M_{i,k}^{(\tilde{\pi})}(t)=\mathsf{MW}_{\tilde{\pi}}(t),
\end{eqnarray} 
and the result follows.

%
%


\textit{Case 2}: Condition \textbf{C2} is satisfied at time $t$. In this case, by using policy $\pi$ at time slot $t$ queue $n$ is receiving service but queue $m$ is not. In contrast, by using policy $\tilde{\pi}$, at time slot $t$ queue $m$ is receiving service but queue $n$ is not. The service of other queues is not disturbed, i.e., the other queues which were receiving service by policy $\pi$ still receive a service by policy $\tilde{\pi}$ at time slot $t$ and the ones that were not receiving service under policy $\pi$ still do not get service under policy $\tilde{\pi}$.
Therefore, 
\begin{eqnarray}
\lefteqn{\mathsf{MW}_{\tilde{\pi}}(t) - \mathsf{MW}_{\pi}(t) =\sum_{\substack{i=1\\i\neq m,n}}^N x_i(t-1) \sum_{k=1}^{K}c_{i,k}(t) M_{i,k}^{(\tilde{\pi})}(t)+x_m(t-1)} \notag \\&& - \sum_{\substack{i=1\\i\neq m,n}}^N x_i(t-1) \sum_{k=1}^{K}c_{i,k}(t) M_{i,k}^{(\pi)}(t)-x_n(t-1)= x_m(t-1) - x_n(t-1) >0~~~~~
\end{eqnarray}
Therefore, $\mathsf{MW}_{\pi}(t) < \mathsf{MW}_{\tilde{\pi}}(t)$.
\end{proof}
\vspace{-3mm}
\subsection{Proof of Lemma \ref{l2}} \label{l13}
\begin{proof}
Without loss of generality, we may convert the bipartite graph $G_t$ to a complete weighted bipartite graph $G'_t$ with $\max\{N,K\}$ vertices in each part. This is done by adding some vertices and edges of zero weight as necessary. In particular, if $N>K$, we will add $N-K$ servers on the right hand side with edges of weight zero to each queue (each vertex on the left hand side). 
If $N<K$, we will add $K-N$ queues on the left hand side with edges of weight zero to each server (each vertex on the right hand side). 
This will not change the operation of the system since the added queues and servers are disconnected from the whole system. 
We denote the sets of vertices on each part of $G'_t$ by $\mathcal{N}'$ and $\mathcal{K}'$, respectively and the set of edges by $\mathcal{E}'$. Consequently, a policy $\pi$ is defined as $\pi=\{\mat{M}^{(\pi)}(t)\}_{t=1}^{\infty}$ where $\mat{M}^{(\pi)}(t)$ is a \textit{perfect matching}\footnote{A perfect matching is a matching that matches all vertices of the graph.} in the complete bipartite graph $G'_t$. We can easily verify that a maximum weighted perfect matching $ \mat{M}^{(\text{MWM})}(t) $ in the complete bipartite graph $G'_t$ is the same as the maximum weighted matching in graph $G_t$ if we remove the added edges of weight zero from matching $\mat{M}^{(\text{MWM})}(t)$.

Consider a policy $\pi$ which is employing perfect matching $\mat{M}^{(\pi)}(t)$ at time slot $t$. Suppose that $\mat{M}^{(\pi)}(t)$ is not a maximum weighted perfect matching on graph $G'_t$, i.e., $\mathsf{MW}_{\pi}(t) <\mathsf{MW}_{\text{MWM}}(t)$. Also, consider a maximum weighted perfect matching $\mat{M}^{(\text{MWM})}(t)$ at time slot $t$. Now, consider these two matchings on $G'_t=(\mathcal{N}',\mathcal{K}',\mathcal{E}')$. Each of $\mat{M}^{(\pi)}(t)$ and $\mat{M}^{({\text{MWM}})}(t)$ corresponds to a distinct sub-graph of $G'_t$ namely ${G'}_t^{(\pi)}=(\mathcal{N}',\mathcal{K}',\mathcal{E}'^{(\pi)})$ and ${G'}_t^{({\text{MWM}})}=(\mathcal{N}',\mathcal{K}',\mathcal{E}'^{({\text{MWM}})})$, respectively. We now build two \textit{\textbf{directed}}, weighted sub-graphs $D_t^{(\pi)}$ and $D_t^{({\text{MWM}})}$ as follows: $D_t^{(\pi)}$ is the same as  ${G'}_t^{(\pi)}$ with all the edges directed from $\cal N'$ to $ \cal K'$ with the same edge weights as ${G'}_t^{(\pi)}$. $D_t^{({\text{MWM}})}$ is the same as  ${G'}_t^{({\text{MWM}})}$ with all the edges directed from $\cal K'$ to $ \cal N'$ with edge weights equal to the \textit{negative} of edge weights of ${G'}_t^{({\text{MWM}})}$. 
Now, consider graph $ U=D_t^{(\pi)} \bigcup D_t^{({\text{MWM}})}$, i.e., the union of the sub-graphs $ D_t^{(\pi)}$ and $D_t^{({\text{MWM}})}$. The graph $U$ can be seen as the union of a number of even cycles\footnote{A cycle with even number of vertices.} denoted by $L$. This is directly concluded from the fact that $ D_t^{(\pi)}$ and $D_t^{({\text{MWM}})}$ are each perfect matchings of $G'_t$ and thus each vertex is incident to an incoming edge and an outgoing edge. Furthermore, for the weight of $U$ shown by $w(U)$, we have 
\begin{eqnarray} \label{weightU}
w(U)=\sum_{\ell \in L}w(\ell) = \mathsf{MW}_{\pi}(t) - \mathsf{MW}_{\text{MWM}}(t)<0.
\end{eqnarray}
In (\ref{weightU}), $w(\ell)$ is the weight of edge $\ell$ in $L$. Therefore, there must exist a negative cycle\footnote{A cycle whose total edge weight is negative.} in $U$. We denote this negative cycle by $\ell^{\star}$. The cycle $\ell^{\star}$ is an even cycle and contains an even number of nodes and edges. We assume that $\ell^{\star}$ contains $2W$ nodes ($W$ nodes from $\cal N'$ and $W$ nodes from $\cal K'$) and also $2W$ edges. Let us denote the nodes of sets $\cal N'$ and $\cal K'$ that form $\ell^\star$ by $n_1,n_2,\dots,n_W$ and $k_1,k_2,\dots,k_W$, respectively. Thus, the cycle $\ell^\star$ can be represented by the sequence of its edges as $\ell^\star=e_{n_1,k_1},e_{k_1,n_2},e_{n_2,k_2},e_{k_2,n_3},\dots,e_{k_{W-1},n_W},$ $e_{n_{W},k_W},
e_{k_{W},n_1} $ (see Figure \ref{cycle}).
\begin{figure}[tp]
    \centering
    \psfrag{a}[][][.9]{$n_1$}
    \psfrag{b}[][][.9]{$n_2$}
    \psfrag{c}[][][.9]{$n_W$}
    \psfrag{d}[][][.9]{$k_1$}
    \psfrag{e}[][][.9]{$k_2$}
    \psfrag{f}[][][.9]{$k_W$}
    \includegraphics[width=.26\textwidth]{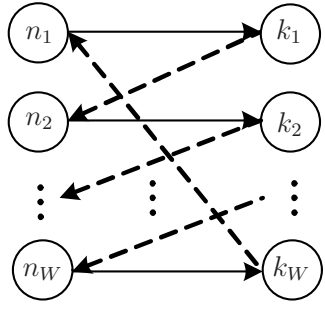}
    \caption{The negative cycle $\ell^\star$}\vspace{-7mm}
    \label{cycle}
\end{figure} 

The edges $e_{k_1,n_2},e_{k_2,n_3},\dots,e_{k_W,n_1}$ belonging to $ D_t^{({\text{MWM}})} $ have negative weights while the edges $e_{n_1,k_1},e_{n_2,k_2},\dots,e_{n_W,k_W}$ belonging  to $ D_t^{(\pi)} $ have positive weights. We can also represent the edges of $\ell^\star$ by pairing the edges incident to each node $n_1,n_2,\dots,n_W$ as follows;
$\ell^\star=(e_{k_W,n_1},e_{n_1,k_1}),$ $(e_{k_1,n_2},e_{n_2,k_2}),\dots,(e_{{k_{W-1},n_W}},e_{n_{W},k_{W}})$. 
We label each pair of edges as follows:
 $r_{n_1}=(e_{k_W,n_1},e_{n_1,k_1}),$ $r_{n_2}=(e_{k_1,n_2},e_{n_2,k_2}),\dots,r_{n_W}=
(e_{{k_{W-1},n_W}},e_{n_{W},k_{W}})$. All $ r_{n_1},r_{n_2},\dots,r_{n_W} $ pairs in $\ell^\star$ have the following property:
The weights associated to the edges of each pair $r_{n_i}$, $1\leq i \leq W$ is either $(0,0)$, $(0,x_{n_i}(t-1))$,$(-x_{n_i}(t-1),0)$ or $(-x_{n_i}(t-1),x_{n_i}(t-1))$. Accordingly, we will specify three types of edge pairs as follows:
\begin{itemize}
\item \textbf{Type 0 (T0)}: pairs with edge weights $(0,0)$ or $(-x_{n_i}(t-1),x_{n_i}(t-1))$, $1\leq i \leq W$ .
\item \textbf{Type 1 (T1)}: pairs with edge weights $(-x_{n_i}(t-1),0)$, $1\leq i \leq W$ and $x_{n_i}(t-1)>0$.
\item \textbf{Type 2 (T2)}: pairs with edge weights $(0,x_{n_i}(t-1))$, $1\leq i \leq W$ and $x_{n_i}(t-1)>0$.
\end{itemize}
%

Since $\ell^{\star}$ is a negative cycle, there must exist a node $n_j$ with edge pairs $(e_{k_{j-1},n_j},e_{n_j,k_j})=(-x_{n_j}(t-1),0)$ and $ x_{n_j}(t-1)>0 $ (if $j=1$, $(e_{k_{W},n_1},e_{n_1,k_1})=(-x_{n_1}(t-1),0)$ and $x_{n_1}(t-1)>0$). Obviously, the edge pair associated with node $n_j$ is of type \textbf{T1}.
We now consider the following three, exhaustive cases:
\begin{enumerate}[]
\item \textit{Case 1}: All the edge pairs other than $r_{n_j}$ are of type \textbf{T0}. In this case, by replacing the edges of $ D_t^{(\pi)} $ by the edges of $ D_t^{({\text{MWM}})} $ in the negative cycle $\ell^\star$ we obtain a server reallocation at time slot $t$. This server reallocation is balancing as in $\pi$ queue $n_{j}$ was not receiving service while after the server reallocation it is. Therefore, if $\tilde{\mat{x}}'(t)$ is the queue length vector after server reallocation we have $\tilde{\mat{x}}'(t) \leq \mat{x}'(t)$ and $\tilde{x}'_{n_j}(t)<{x}'_{n_j}(t)$ (satisfying condition \textbf{C1}).

\item \textit{Case 2}: If we trace backward the edge pairs of cycle $\ell^\star$ from $r_{n_{j}}$ and the first non-\textbf{T0} pair is a \textbf{T1} pair. Let $r_{n_{j'}}$ be such a \textbf{T1} edge pair. In other words, in cycle $\ell^{\star}$ the pairs $r_{n_{j'}}$ and $r_{n_{j}}$ are of type \textbf{T1} while other pairs \textit{between them} are of type \textbf{T0} (see Figure \ref{cases:2}).
In this case, by replacing the edges of $ D_t^{(\pi)} $ by the edges of $ D_t^{({\text{MWM}})} $ just for nodes $n_{j'+1}$ to $n_{j}$ of cycle $\ell^\star$ and allocating $k_j$ to $n_{j'}$ we obtain a server reallocation at time slot $t$. This server reallocation is balancing as in $\pi$ queue $n_{j}$ does not receive service while after the server reallocation it does. Furthermore, queue $n_{j'}$ was not being served in $\pi$ but 
under the new server reallocation it may get service (depending on the weight of edge $e_{k_{j},n_{j'}}$) and the service of other queues is not disturbed. Therefore, if $\tilde{\mat{x}}'(t)$ is the intermediate queue length vector after server reallocation, we have $\tilde{\mat{x}}'(t) \leq \mat{x}'(t)$ and $\tilde{x}'_{n_j}(t)<x'_{n_j}(t)$ (satisfying condition \textbf{C1}).
 
\begin{figure}
  \centering
  	\psfrag{l}[][][1]{$0$}
    \psfrag{m}[][][.9]{$x_{n_{j'}}(t-1)$}
    \psfrag{n}[][][.9]{$-x_{n_j}(t-1)$}
    \psfrag{o}[][][1]{$0$}
    \psfrag{p}[][][1]{$k_{j'}$}
    \psfrag{q}[][][1]{$k_j$}
    \psfrag{r}[][][1]{$n_{j'}$}
    \psfrag{s}[][][1]{$n_j$}
    \psfrag{t}[][][1]{$T_2$}
    \psfrag{u}[][][1]{$T_0$}
    \psfrag{v}[][][1]{$T_1$}
    \psfrag{a}[][][.9]{$-x_{n_{j'}}(t-1)$}
    \psfrag{b}[][][1]{$0$}
    \psfrag{c}[][][.9]{$-x_{n_{j}}(t-1)$}
    \psfrag{d}[][][1]{$0$}
    \psfrag{e}[][][1]{$k_{j'}$}
    \psfrag{f}[][][1]{$k_j$}
    \psfrag{g}[][][1]{$n_{j'}$}
    \psfrag{h}[][][1]{$n_j$}
    \psfrag{i}[][][1]{$T_1$}
    \psfrag{j}[][][1]{$T_0$}
    \psfrag{k}[][][1]{$T_1$}
  \subfloat[Case 2]{\label{cases:2}\includegraphics[width=0.33\textwidth]{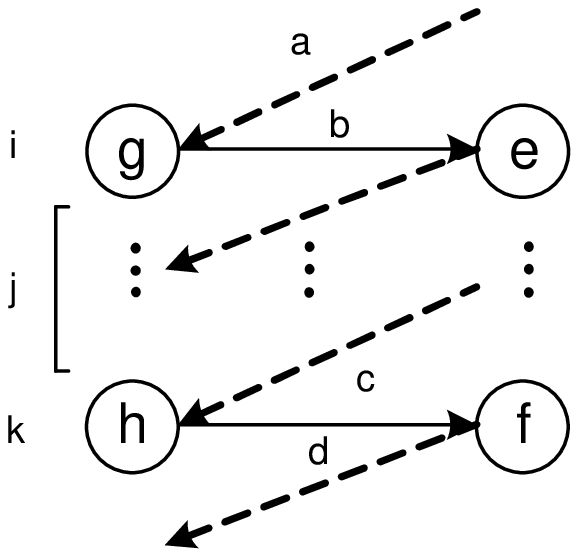}}              
  \subfloat[Case 3]{\label{cases:3}\includegraphics[width=0.33\textwidth]{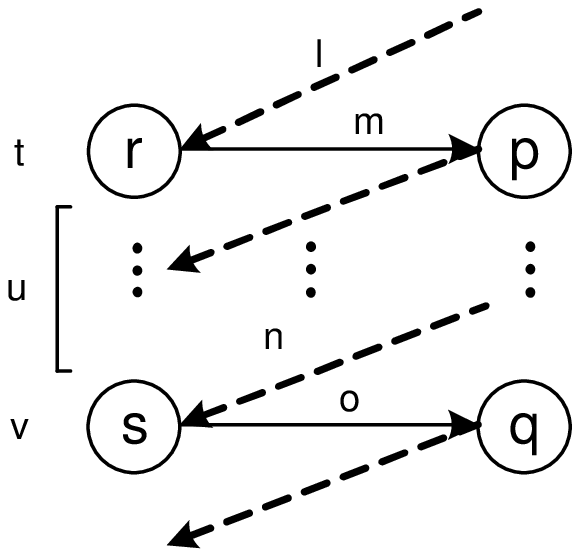}}
  \caption{Cases 2 and 3 in the proof of Lemma \ref{l2}}\vspace{-7mm}
  \label{examples_l2}
\end{figure}
\item \textit{Case 3}: If we go backward on the edge pairs of cycle $\ell^\star$ from $r_{n_{j}}$ and the first non-\textbf{T0} pair is a \textbf{T2} pair. Let $r_{n_{j'}}$ be such a \textbf{T2} edge pair. In other words, in cycle $\ell^{\star}$, the pair $r_{n_{j'}}$ is of type \textbf{T2} and $r_{n_{j}}$ is of type \textbf{T1} while other pairs between them are of type \textbf{T0} (see Figure \ref{cases:3}).
First of all, we claim that $x_{n_{j'}}(t-1)\leq x_{n_{j}}(t-1)$. If $x_{n_{j'}}(t-1)> x_{n_{j}}(t-1)$, by replacing the edges of $ D_t^{({\text{MWM}})} $ by the edges of $ D_t^{(\pi)} $ just for nodes $n_{j'}$ to $n_{j-1}$ of the cycle $\ell^\star$ and not serving queue $n_{j}$, without disturbing the service of other queues we obtain a server reallocation at time slot $t$ with larger $\mathsf{MW}$ index than $\mathsf{MW} _{\text{MWM}}(t)$ and this contradicts the fact that MWM has the maximum $\mathsf{MW}$ index at time slot $t$.
Accordingly, we consider the following two sub-cases:
\begin{enumerate}[]
\item \textit{Sub-case 3.1}: $x_{n_{j}}(t-1)>x_{n_{j'}}(t-1)$. In this case, we replace the edges of $ D_t^{(\pi)} $ by the edges of $ D_t^{({\text{MWM}})} $ just for queues $n_{j'+1}$ to $n_{j}$ of the cycle $\ell^\star$ and allocate server $k_{j}$ to queue $n_{j'}$. Server $k_{j}$ may or may not serve queue $n_{j'}$. We consider the worst case where it does not serve this queue ($w_{n_{j'},k_{j}}=0$). Thus, without disturbing the service of other queues we obtain a server reallocation with the following property: all the queues other than $n_{j}$ and $n_{j'}$ have the same service as before; queue $n_{j}$ which was receiving zero service in $\pi$, now receives one packet service and queue $n_{j'}$ which was receiving one packet service, now in the worst case loses its service. Therefore, this server reallocation is balancing as the queue length after server reallocation (denoted by $\tilde{\mat{x}}'(t)$) and the queue length after applying policy $\pi$ (denoted by $x'(t)$) satisfy condition \textbf{C2}. So, $\tilde{\mat{x}}'(t) \preceq \mat{x}'(t)$ and therefore the applied server reallocation is a balancing one.
\item \textit{Sub-case 3.2}: $x_{n_{j}}(t-1)=x_{n_{j'}}(t-1)$. In this case, the sequence of edge pairs $r_{n_{j'}},\dots,r_{n_{j}}$ contribute ``0'' in the calculation of $w(U)$. Therefore, we may treat them as a sequence of edge pairs of type \textbf{T0}. By doing so, we have a new negative cycle $\ell'^\star$ which is the same as $\ell^\star$ but the sequence of edge pairs $r_{n_{j'}},\dots,r_{n_{j}}$ are replaced by the same number of edge pairs of type \textbf{T0}. Cycle $\ell'^\star$ is a negative cycle with the same weight as $\ell^\star$.  Therefore, there must exist an edge pair of type \textbf{T1} in $\ell'^\star$ (In this case, it is not possible for $\ell'^\star$ to have just edge pairs of type \textbf{T0} as it results in zero weight for the cycle) and all the cases of 1, 2 and 3 apply for $\ell'^\star$ as well. As long as case 3.2 is true, we always obtain a new negative cycle for which we will have one of the cases 1,2 and 3 satisfied where we can determine a balancing server reallocation. 
\end{enumerate}
\end{enumerate}
Cases 1, 2 and 3 cover all the possible instances of negative cycle $\ell^\star$ for all of which we proved that there exists a balancing server reallocation. In summary, we proved that if the policy $\pi$ does not employ a maximum weighted matching at time slot $t$, there exists a negative cycle in $U$, the graph obtained by taking the union of the matchings of $\pi$ and MWM. We proved that by reallocation of the servers involved in the negative cycle, we always can find a balancing server reallocation for policy $\pi$.
\end{proof}


\section{Proof of Lemmas \ref{l3} and \ref{lnew}  } 


\subsection{Proof of Lemma \ref{l3}} \label{l14}
\begin{proof}
Fix any arbitrary policy $ \pi \in \Pi_t^h$ where $h=h_t^{\pi}>0$, and any arbitrary sample path $\mat{\omega}=(\mat{x}(0),\mat{c}(1),\mat{a}(1),$ $\mat{x}(1),\mat{c}(2),\mat{a}(2),\mat{x}(2),\dots)$ of the underlying random variables $(\mat{X}(0),\mat{C}(1),$ $\mat{A}(1),\mat{X}(1),\mat{C}(2),$ $\mat{A}(2),\mat{X}(2), \dots)$. We apply the coupling method to construct from $\mat{\omega}$ a new sample path $\tilde{\mat{\omega}} = (\tilde{\mat{x}}(0), \tilde{\mat{c}}(1), \tilde{\mat{a}}(1), \tilde{\mat{x}}(1),$ $\tilde{\mat{c}}(2), \tilde{\mat{a}}(2),\tilde{\mat{x}}(2), \dots)$ resulting in a new sequence of random variables  
$(\tilde{\mat{X}}(0),$ $\tilde{\mat{C}}(1), \tilde{\mat{A}}(1), \tilde{\mat{X}}(1), \tilde{\mat{C}}(2), \tilde{\mat{A}}(2),\tilde{\mat{X}}(2),\dots)$\hspace{1pt}with
$\mat{X}(0)=\tilde{\mat{X}}(0)$. Recall that $\mat{X}(0)$ is the queue length vector in which the system starts. We denote the policy defined on the new sample path $\tilde{\mat{\omega}}$ by $\tilde{\pi}$. In fact, we construct $\tilde{\mat{\omega}}$ and $\tilde{\pi}\in \Pi_t^{h-1}$ in such a fashion that \textit{for all} the sample paths we have $\tilde{\mat{x}}(t') \preceq_{p} \mat{x}(t')$ for all $t' = 1,2,\dots$.
Therefore, for any $f \in \cal F$ we have $f(\tilde{\mat{x}}(t')) \leq f(\mat{x}(t'))$ for all $t'$. As it will be shown, the processes $\{(\mat{C}(t'),\mat{A}(t'))\}_{t'=1}^{\infty}$ and $\{(\tilde{\mat{C}}(t'),\tilde{\mat{A}}(t'))\}_{t'=1}^{\infty}$ are the same in distribution (these processes are permutation invariant). Thus, the process $f(\tilde{\mat{X}})=\{f(\tilde{\mat{X}}(t'))\}_{t'=1}^{\infty} $ obtained by applying policy $\tilde{\pi}$ to the system is stochastically smaller than $f({\mat{X}})= \{f(\mat{X}(t'))\}_{t'=1}^{\infty}$, i.e., $f(\tilde{\mat{X}}) \leq_{st}  f(\mat{X})) $ and $\tilde{\pi}$ dominates $\pi$.
%
%
%

%
%
%

Therefore, in the following, our goal will be to construct $\tilde{\pi}$ and $\tilde{\mat{\omega}}$ such that $\tilde{\mat{x}}(t') \preceq _p \mat{x}(t')$ for all time slots.
In the proof, we always use the tilde notation for all random variables that belong to the new system. The construction of $\tilde{\pi}$ is done in two steps:

\textbf{Step 1: Construction of $\tilde{\pi}$ for $\tau \leq t$}: To construct the new sample path $\tilde{\mat{\omega}}$ we let the arrival, connectivity and the policy be the same as the first system until time slot $t-1$ , i.e., $\tilde{\mat{c}}(\tau) =\mat{c}(\tau)$, $\tilde{\mat{a}}(\tau) =\mat{a}(\tau)$ and $\mat{M}^{(\pi)}(\tau)=\mat{M}^{(\tilde{\pi})}(\tau)$ for $\tau \leq t-1$. Thus, the resulting queue lengths at the beginning of time slot $t$ (or at the end of time slot $t-1$) are equal, i.e., $\tilde{\mat{x}}(t-1)=\mat{x}(t-1)$. 

We now consider the construction of $\tilde{\mat{\omega}}$ and $\tilde{\pi}$ for time slot $t$. 
Since $\pi \in \Pi^{h}_{t}$ and $h>0$, according to Lemma \ref{l2} there exists a balancing server reallocation such that either \textbf{C1} or \textbf{C2} is satisfied. Thus, we consider two cases:

\textit{Case 1}: After applying the balancing server reallocation, condition \textbf{C1} is satisfied. In other words, there exists a matching such that if applied on the queue length $\tilde{\mat{x}}(t-1)=\mat{x}(t-1)$ at time slot $t$, we get $\tilde{\mat{x}}'(t)$ such that $\tilde{\mat{x}}'(t) \leq \mat{x}'(t)$. We denote such a matching by $\mat{M}^{(\tilde{\pi})}(t)$.
In this case, we let $\tilde{\mat{c}}(t)=\mat{c}(t)$ and $\tilde{\mat{a}}(t)=\mat{a}(t)$ and we apply $\mat{M}^{(\tilde{\pi})}(t)$ at time slot $t$, i.e., arrivals and connectivities are the same in both systems and policy $\tilde{\pi}$ acts at time slot $t$. So, we can easily check that $\tilde{\mat{x}}(t) \leq \mat{x}(t)$ and therefore $\tilde{\mat{x}}(t) \preceq \mat{x}(t)$.

\textit{Case 2}: After applying the balancing server reallocation, condition \textbf{C2} is satisfied. In other words, there exists a matching such that if applied on the system at time slot $t$, we get $\tilde{\mat{x}}'(t)$ which is different from $\mat{x}'(t)$ in two elements $m$ and $n$ such that $x'_n(t) <\tilde{x}'_n(t)\leq \tilde{x}'_m(t)< x'_m(t)$ and the following constraints are satisfied: $\tilde{x}'_n(t)=x'_n(t)+1$ and $\tilde{x}'_m(t)=x'_m(t)-1$.
We call such a matching by $\mat{M}^{(\tilde{\pi})}(t)$. In this case, we let $\tilde{\mat{c}}(t)=\mat{c}(t)$ and $\tilde{\mat{a}}(t)=\mat{a}(t)$ and we apply $\mat{M}^{(\tilde{\pi})}(t)$ at time slot $t$, i.e., arrivals and connectivities are the same in both systems and policy $\tilde{\pi}$ acts at time slot $t$. We consider all the following conditions for arrivals to queues $m$ and $n$ as follows:
\begin{itemize}
\item If there is no arrival or there is an arrival to both queues $m$ and $n$ (i.e., $a_m(t)=a_n(t)=0$ or $a_m(t)=a_n(t)=1$), we conclude that $\tilde{\mat{x}}(t)$ and $ \mat{x}(t)$ satisfy condition \textbf{D3}. Thus, $\tilde{\mat{x}}(t) \preceq \mat{x}(t)$.
\item  If there is an arrival to queue $m$ but not $n$ (i.e., $a_m(t)=1$, $a_n(t)=0$), we conclude that $\tilde{\mat{x}}(t)$ and $ \mat{x}(t)$ satisfy condition \textbf{D3}. Thus, $\tilde{\mat{x}}(t) \preceq \mat{x}(t)$.
\item If there is an arrival to queue $n$ but not $m$ (i.e., $a_m(t)=0$, $a_n(t)=1$) and $\tilde{x}_m(t)=\tilde{x}_n(t)$, we conclude that $\tilde{\mat{x}}(t)$ and $ \mat{x}(t)$ satisfy condition \textbf{D2}. Thus, $\tilde{\mat{x}}(t) \preceq \mat{x}(t)$.
\item  If there is an arrival to queue $n$ but not $m$ (i.e., $a_m(t)=0$, $a_n(t)=1$) and $\tilde{x}_n(t)<\tilde{x}_m(t)$, we conclude that $\tilde{\mat{x}}(t)$ and $ \mat{x}(t)$ satisfy condition \textbf{D3}. Thus, $\tilde{\mat{x}}(t) \preceq \mat{x}(t)$. 
\end{itemize} 
In all the cases we can see that $\tilde{\mat{x}}(t) \preceq \mat{x}(t)$.
The obtained policy $\tilde{\pi}$ belongs to $\Pi_t^{h-1}$ since we applied a balancing server reallocation to the matching employed in $\pi$ at time slot $t$.

\textbf{Step 2: Construction of $\tilde{\pi}$ for $\tau > t$}:
In this step, we focus on construction of $\tilde{\mat{\omega}}$ and $\tilde{\pi}$ for $\tau > t$. We will employ mathematical induction to achieve this goal. In particular, we assume that $\tilde{\mat{\omega}}$ and $\tilde{\pi}$ are constructed up to time slot $\tau$ ($\tau \geq t$) such that $\tilde{\mat{x}}(\tau) \preceq \mat{x}(\tau)$ i.e., one of the conditions \textbf{D1}, \textbf{D2} and \textbf{D3} is satisfied for $\mat{x}(\tau)$ and $\tilde{\mat{x}}(\tau)$. We will prove that policy $\tilde{\pi}$ and sample path $\tilde{\mat{\omega}}$ can be constructed such that $\tilde{\mat{x}}(\tau+1) \preceq \mat{x}(\tau+1)$ (i.e., one of the conditions \textbf{D1}, \textbf{D2} and \textbf{D3} is satisfied for $\mat{x}(\tau+1)$ and $\tilde{\mat{x}}(\tau+1)$). Accordingly, we consider three cases corresponding to each condition \textbf{D1}, \textbf{D2} or \textbf{D3} at time slot $\tau$: 

\textit{Case 1}: $\tilde{\mat{x}}(\tau) \leq \mat{x}(\tau)$. In this case, the construction of $\tilde{\mat{\omega}}$ and $\tilde{\pi}$ at time slot $\tau+1$ is straightforward. We let $\tilde{\mat{c}}(\tau+1)=\mat{c}(\tau+1)$ and $\tilde{\mat{a}}(\tau+1)=\mat{a}(\tau+1)$ and $\mat{M}^{(\tilde{\pi})}(\tau+1)=\mat{M}^{(\pi)}(\tau+1) $. Thus, $\tilde{\mat{x}}(\tau+1) \leq \mat{x}(\tau+1)$ and therefore $\tilde{\mat{x}}(\tau+1) \preceq \mat{x}(\tau+1)$.

\textit{Case 2}: $\tilde{\mat{x}}(\tau)$ is obtained from $\mat{x}(\tau)$ by permutation of two distinct elements $m$ and $n$. In this case, we let $\tilde{c}_{n,k}(\tau+1)=c_{m,k}(\tau+1)$ and $\tilde{c}_{m,k}(\tau+1)=c_{n,k}(\tau+1)$ for $k=1,2,\dots,K$; $\tilde{c}_{i,k}(\tau+1)=c_{i,k}(\tau+1)$ for all $i \in \cal N$, $i\neq n,m$ and $k=1,2,\dots,K$; $\tilde{a}_n(\tau+1)=a_m(\tau+1)$, $\tilde{a}_m(\tau+1)=a_n(\tau+1)$ and $\tilde{a}_i(\tau+1)=a_i(\tau+1)$ for $i \in \cal N$, $i\neq n,m$. Suppose that $\mat{M}^{(\pi)}(\tau+1)=(M^{(\pi)}_{n,k}(\tau+1))~\forall n \in \mathcal{N}, k \in \mathcal{K}$ be the employed matching by policy $\pi$ at time slot $\tau+1$. We construct $\mat{M}^{(\tilde{\pi})}(\tau+1)$ as follows: Let $M^{(\tilde{\pi})}_{i,k}(\tau+1)=M^{(\pi)}_{i,k}(\tau+1)$ for $i \in \cal N$, $i\neq n,m$ and also let $M^{(\tilde{\pi})}_{n,k}(\tau+1)=M^{(\pi)}_{m,k}(\tau+1)$ and $M^{(\tilde{\pi})}_{m,k}(\tau+1)=M^{(\pi)}_{n,k}(\tau+1)$ for $k=1,2,\dots,K$. As a result, $\tilde{\mat{x}}(\tau+1)$ and $\mat{x}(\tau+1)$ satisfy condition \textbf{D2} at time slot $\tau+1$ and therefore $\tilde{\mat{x}}(\tau+1) \preceq \mat{x}(\tau+1) $.

\textit{Case 3}: $\tilde{\mat{x}}(\tau)$ is obtained from $\mat{x}(\tau)$ by performing a balancing interchange of two distinct elements $m$ and $n$ as defined in condition \textbf{D3}. In particular, $\tilde{\mat{x}}(\tau)$ and $\mat{x}(\tau)$ are different in only two elements $n$ and $m$ such that $x_n(\tau) <\tilde{x}_n (\tau)\leq\tilde{x}_m (\tau)< x_m(\tau)$ and the following constraints are satisfied: $\tilde{x}_n(\tau)=x_n(\tau)+1$ and $\tilde{x}_m(\tau)=x_m(\tau)-1$.
In this case, we consider the following sub-cases:

\textit{Sub-case 3.1}: $\tilde{x}_n (\tau) < \tilde{x}_m (\tau)-1$: In this case, we let $\tilde{\mat{c}}(\tau+1)=\mat{c}(\tau+1)$ and $\tilde{\mat{a}}(\tau+1)=\mat{a}(\tau+1)$ and we let $\mat{M}^{(\tilde{\pi})}(\tau+1)=\mat{M}^{(\pi)}(\tau+1)$. Thus, if $x_n(\tau)=0$ and queue $n$ is served, condition \textbf{D1} is satisfied at $\tau+1$. Otherwise, $\tilde{\mat{x}}(\tau+1)$ is obtained from $\mat{x}(\tau+1)$ by performing a balancing interchange of elements $m$ and $n$. Therefore $\tilde{\mat{x}}(\tau+1) \preceq \mat{x}(\tau+1) $.

\textit{Sub-case 3.2:} $\tilde{x}_n (\tau) = \tilde{x}_m (\tau)-1$:  In this case again we let $\tilde{\mat{c}}(\tau+1)=\mat{c}(\tau+1)$ and $\tilde{\mat{a}}(\tau+1)=\mat{a}(\tau+1)$ and we let $\mat{M}^{(\tilde{\pi})}(\tau+1)=\mat{M}^{(\pi)}(\tau+1)$. Thus, if $x_n(\tau)=0$ and queue $n$ is served, condition \textbf{D1} is satisfied at $\tau+1$. If queue $m$ gets service, queue $n$ does not get service, there is an arrival to queue $n$ and no arrival to queue $m$, then $\tilde{x}_n (\tau+1)=x_m (\tau+1)$ and $\tilde{x}_m (\tau+1)=x_n (\tau+1)$. Therefore, $ \tilde{\mat{x}}(\tau+1)$ and $\mat{x}(\tau+1)$ satisfy condition \textbf{D2} and $\tilde{\mat{x}}(\tau+1) \preceq \mat{x}(\tau+1) $. Otherwise, $ \tilde{\mat{x}}(\tau+1)$ and $\mat{x}(\tau+1)$ satisfy condition \textbf{D3} and $\tilde{\mat{x}}(\tau+1) \preceq \mat{x}(\tau+1) $.

\textit{Sub-case 3.3}: $\tilde{x}_n (\tau) = \tilde{x}_m (\tau)$:
In this case, we let $\tilde{\mat{c}}(\tau+1)=\mat{c}(\tau+1)$ and $\mat{M}^{(\tilde{\pi})}(\tau+1)=\mat{M}^{(\pi)}(\tau+1)$. Now, we consider the following cases to determine the arrivals at time slot $\tau+1$.
\begin{itemize}
\item If $x_n(\tau)>0$ and both queues $m$ and $n$ or none of them get service at time slot $\tau+1$, we let $\tilde{\mat{a}}(\tau+1)=\mat{a}(\tau+1)$. Therefore, if $a_m(\tau+1)=0$ and $a_n(\tau+1)=1$, $\tilde{\mat{x}}(\tau+1)$ and $\mat{x}(\tau+1)$ satisfy condition \textbf{D2} and thus $\tilde{\mat{x}}(\tau+1) \preceq \mat{x}(\tau+1)$. Otherwise, $\tilde{\mat{x}}(\tau+1)$ and $\mat{x}(\tau+1)$ satisfy condition \textbf{D3} and thus $\tilde{\mat{x}}(\tau+1) \preceq \mat{x}(\tau+1)$.

\item If $x_n(\tau)>0$ and queue $n$ gets service at time slot $\tau+1$ and queue $m$ does not get service at time slot $\tau+1$, we let $\tilde{\mat{a}}(\tau+1)=\mat{a}(\tau+1)$. Therefore, $\tilde{\mat{x}}(\tau+1)$ and $\mat{x}(\tau+1)$ satisfy condition \textbf{D3} and thus $\tilde{\mat{x}}(\tau+1) \preceq \mat{x}(\tau+1)$.

\item If $x_n(\tau)>0$ and queue $m$ gets service at time slot $\tau+1$ and queue $n$ does not get service at time slot $\tau+1$, we let $\tilde{a}_m(\tau+1)=a_n(\tau+1)$ and $\tilde{a}_n(\tau+1)=a_m(\tau+1)$ and $\tilde{a}_i(\tau+1)=a_i(\tau+1)$ for $i\in\cal N$ and $i \neq m,n$. Therefore, $\tilde{\mat{x}}(\tau+1)$ and $\mat{x}(\tau+1)$ satisfy condition \textbf{D2} and thus $\tilde{\mat{x}}(\tau+1) \preceq \mat{x}(\tau+1)$.

\item If $x_n(\tau)=0$ and queue $n$ gets service at time slot $\tau+1$ (although it does not have any packet to be served), we let $\tilde{\mat{a}}(\tau+1)=\mat{a}(\tau+1)$. Therefore, $\tilde{\mat{x}}(\tau+1)$ and $\mat{x}(\tau+1)$ satisfy condition \textbf{D1} and thus $\tilde{\mat{x}}(\tau+1) \preceq \mat{x}(\tau+1)$.

\item If $x_n(\tau)=0$ and queue $m$ gets service at time slot $\tau+1$ and queue $n$ does not get service at time slot $\tau+1$, we let $\tilde{a}_m(\tau+1)=a_n(\tau+1)$ and $\tilde{a}_n(\tau+1)=a_m(\tau+1)$ and $\tilde{a}_i(\tau+1)=a_i(\tau+1)$ for $i\in\cal N$ and $i \neq m,n$. Therefore, $\tilde{\mat{x}}(\tau+1)$ and $\mat{x}(\tau+1)$ satisfy condition \textbf{D2} and thus $\tilde{\mat{x}}(\tau+1) \preceq \mat{x}(\tau+1)$.

\item If $x_n(\tau)=0$ and neither queue $m$ nor $n$ gets service at time slot $\tau+1$, we let $\tilde{\mat{a}}(\tau+1)=\mat{a}(\tau+1)$. If $a_m(\tau+1)=0$ and $a_n(\tau+1)=1$, $\tilde{\mat{x}}(\tau+1)$ and $\mat{x}(\tau+1)$ satisfy condition \textbf{D2} and thus $\tilde{\mat{x}}(\tau+1) \preceq \mat{x}(\tau+1)$. Otherwise, $\tilde{\mat{x}}(\tau+1)$ and $\mat{x}(\tau+1)$ satisfy condition \textbf{D3} and thus $\tilde{\mat{x}}(\tau+1) \preceq \mat{x}(\tau+1)$.
\end{itemize} 
The above cases cover all the possible cases for all of which we constructed $\tilde{\mat{\omega}}$ and $\tilde{\pi}$ such that $\tilde{\mat{x}}(\tau+1) \preceq \mat{x}(\tau+1)$.
 
According to steps 1 and 2, from any sample path $\mat{\omega}$ and any arbitrary policy $\pi \in \Pi_t^h $, $h=h_t^\pi>0$, we can construct a sample path $\tilde{\mat{\omega}}$ and a policy $\tilde{\pi} \in \Pi_t^{h-1}$ such that at all time slots we have $\tilde{\mat{x}}(t') \preceq_p \mat{x}(t')$. Therefore, $f(\tilde{\mat{x}}(t')) \leq f(\mat{x}(t'))$. Consequently, the process $f(\tilde{\mat{X}})=\{f(\tilde{\mat{X}}(t'))\}_{t'=1}^{\infty} $ obtained by applying policy $\tilde{\pi}$ to the system is stochastically smaller than $f({\mat{X}}) =\{f(\mat{X}(t'))\}_{t'=1}^{\infty}$, i.e., $f(\tilde{\mat{X}}) \leq_{st}  f(\mat{X}) $ and therefore $\tilde{\pi} \in \Pi_t^{h-1}$ dominates $\pi \in \Pi_t^{h}$.
\end{proof}

\subsection{Proof of Lemma \ref{l3} for the System with Random Service Failures} \label{wsf}
\begin{proof}
The only difference of the proof in this case from the proof of Lemma \ref{l3} is that we have to consider random variables $Q_{n,k}(t)~\forall n \in \mathcal{N},\forall k \in \mathcal{K}$ in our dynamic coupling argument. Therefore, for the arbitrary sample path 
$\mat{\omega}=(\mat{x}(0),\mat{c}(1),\mat{q}(1),\mat{a}(1),\mat{x}(1),\mat{c}(2),\mat{q}(2),\mat{a}(2),$ $\mat{x}(2),...)$,
 we apply the coupling method to construct from $\mat{\omega}$ a new sample path $\tilde{\mat{\omega}} = (\tilde{\mat{x}}(0), $ $\tilde{\mat{c}}(1), \tilde{\mat{q}}(1), \tilde{\mat{a}}(1), \tilde{\mat{x}}(1), \tilde{\mat{c}}(2),\tilde{\mat{q}}(2), \tilde{\mat{a}}(2),\tilde{\mat{x}}(2), ...)$ resulting in a new sequence of random variables $(\tilde{\mat{X}}(0), $ $ \tilde{\mat{C}}(1),\tilde{\mat{Q}}(1), \tilde{\mat{A}}(1), \tilde{\mat{X}}(1), \tilde{\mat{C}}(2),\tilde{\mat{Q}}(2), \tilde{\mat{A}}(2),\tilde{\mat{X}}(2), ...)$ with $\mat{X}(0)=\tilde{\mat{X}}(0)$. We will follow the same approach and consider the same cases. Before we proceed to the details we introduce the following complementary notation. 
Suppose that $s^{(\pi)}_n(t)$ denotes the index of the server assigned to queue $n$ at time slot $t$ by policy $\pi$. Note that $s^{(\pi)}_n(t)\in \cal K$ or it is empty.

\textbf{Step 1: Construction of $\tilde{\pi}$ for $\tau \leq t$}: In this case, as in the proof of Lemma \ref{l3} we first consider $\tau \leq t-1$. For those slots, we let all $\tilde{\mat{c}}(\tau)$, $\tilde{\mat{q}}(\tau)$, $\tilde{\mat{a}}(\tau)$ and $\mat{M}^{(\tilde{\pi})}(\tau)$ to be the same in both systems i.e., $\tilde{\mat{c}}(\tau) =\mat{c}(\tau)$, $\tilde{\mat{q}}(\tau) =\mat{q}(\tau)$, $\tilde{\mat{a}}(\tau) =\mat{a}(\tau)$ and $\mat{M}^{(\tilde{\pi})}(\tau)=\mat{M}^{(\pi)}(\tau)$ and therefore we have $\tilde{\mat{x}}(t-1)=\mat{x}(t-1)$. 

At time slot $t$, 
the distance of policy $\pi$ to $\Pi_{t}$ is $h$ balancing server reallocations. Since $\pi \in \Pi^{h}_{t}$ and $h>0$, according to Lemma \ref{l2} there exists a balancing server reallocation such that either \textbf{C1} or \textbf{C2} is satisfied. Thus, we consider two cases:

\textit{Case 1}: After applying the balancing server reallocation, condition \textbf{C1} is satisfied. In other words, there exists a matching $\mat{M}^{(\tilde{\pi})}(t)$ such that if applied on the queue lengths $\tilde{\mat{x}}(t-1)=\mat{x}(t-1)$, we get $\tilde{\mat{x}}'(t)$ such that $\tilde{\mat{x}}'(t) \leq \mat{x}'(t)$. 
Note that if a non-empty queue is served under $\mat{M}^{(\pi)}(t)$, it should also get service under $\tilde{\pi}$. Otherwise, the condition $\tilde{\mat{x}}'(t) \leq \mat{x}'(t)$ will be violated. 
In this case, we let $\tilde{\mat{c}}(t)=\mat{c}(t)$ and $\tilde{\mat{a}}(t)=\mat{a}(t)$ and we apply $\mat{M}^{(\tilde{\pi})}(t)$ at time slot $t$.
For any non-empty queue $n$ that is served under both policies $\pi$ and $\tilde{\pi}$ we let $\tilde{q}_{n,s^{(\tilde{\pi})}_n(t)}(t)=q_{n,s^{(\pi)}_n(t)}(t)$ and $\tilde{q}_{n,s^{({\pi})}_n(t)}(t)=q_{n,s^{(\tilde{\pi})}_n(t)}(t)$. In other words, we let each non-empty queue $n$ which was being served in both systems experience the same service failure. For other variables $\tilde{q}_{n,k}(t)$ we let $\tilde{q}_{n,k}(t)=q_{n,k}(t)$.
 Then, we can easily check that $\tilde{\mat{x}}(t) \leq \mat{x}(t)$ and therefore $\tilde{\mat{x}}(t) \preceq \mat{x}(t)$.

\textit{Case 2}: After applying the balancing server reallocation, condition \textbf{C2} is satisfied. In other words, there exists a matching such that if applied on the system at time slot $t$, we get $\tilde{\mat{x}}'(t)$ which is different from $\mat{x}'(t)$ in two elements $m$ and $n$ such that $x'_n(t) <\tilde{x}'_n(t)\leq \tilde{x}'_m(t)< x'_m(t)$ and the following constraints are satisfied; $\tilde{x}'_n(t)=x'_n(t)+1$ and $\tilde{x}'_m(t)=x'_m(t)-1$.
We call such a matching by $\mat{M}^{(\tilde{\pi})}(t)$. Note that in this case, each queue (other than queue $m$ and queue $n$) which is (resp. is not) receiving service under $\pi$ is (resp. is not) also receiving service under $\tilde{\pi}$. Queue $n$ is receiving service under $\pi$ and queue $m$ is not. Queue $n$ is not receiving service under $\tilde{\pi}$ and queue $m$ is. In this case, we let $\tilde{\mat{c}}(t)=\mat{c}(t)$ and $\tilde{\mat{a}}(t)=\mat{a}(t)$ and we apply $\mat{M}^{(\tilde{\pi})}(t)$ at time slot $t$. We let $\tilde{q}_{i,s^{(\tilde{\pi})}_i(t)}(t)=q_{i,s^{(\pi)}_i(t)}(t)$ and $\tilde{q}_{i,s^{({\pi})}_i(t)}(t)=q_{i,s^{(\tilde{\pi})}_i(t)}(t)$ for any queue $i$ which is receiving service under both matchings $\mat{M}^{({\pi})}(t)$ and $\mat{M}^{(\tilde{\pi})}(t)$. We also let $\tilde{q}_{m,s^{(\tilde{\pi})}_m(t)}(t)=q_{n,s^{(\pi)}_n(t)}(t)$, $\tilde{q}_{n,s^{({\pi})}_n(t)}(t)=q_{m,s^{(\tilde{\pi})}_m(t)}(t)$ and for other failure variables we let $\tilde{q}_{n,k}(t)=q_{n,k}(t)$. By such coupling of the service success/failure random variables we will consider the following cases for arrivals and service failures:
\begin{itemize}
\item If $q_{n,s^{(\pi)}_n(t)}(t)=0$, we let $\tilde{\mat{a}}(t)=\mat{a}(t)$ and therefore, $\tilde{\mat{x}}(t)=\mat{x}(t)$ which implies $\tilde{\mat{x}}(t)\preceq \mat{x}(t)$.
\item If $q_{n,s^{(\pi)}_n(t)}(t)=1$, then we can use the same coupling argument on the arrival processes as what we did in the proof of Lemma \ref{l3} and conclude that $\tilde{\mat{x}}(t) \preceq \mat{x}(t)$. We omit the argument to avoid redundant discussion.
\end{itemize}

Note that the obtained policy $\tilde{\pi}$ belongs to $\Pi_t^{h-1}$ as we applied a balancing server reallocation to the matching employed in $\pi$ at time slot $t$.

\textbf{Step 2: Construction of $\tilde{\pi}$ for $\tau > t$}:
In this step, the same as the proof of Lemma \ref{l3} we use mathematical induction to construct policy $\tilde{\pi}$ for $\tau > t$. Therefore, suppose that $\tilde{\mat{\omega}}$ and $\tilde{\pi}$ are constructed up to time slot $\tau$ ($\tau \geq t$) such that $\tilde{\mat{x}}(\tau) \preceq \mat{x}(\tau)$. Therefore, one of the conditions \textbf{D1}, \textbf{D2} and \textbf{D3} is satisfied for $\mat{x}(\tau)$ and $\tilde{\mat{x}}(\tau)$ as follows.


\textit{Case 1}: $\tilde{\mat{x}}(\tau) \leq \mat{x}(\tau)$. In this case, the construction of $\tilde{\mat{\omega}}$ and $\tilde{\pi}$ at time slot $\tau+1$ is straightforward. We let $\tilde{\mat{c}}(\tau+1)=\mat{c}(\tau+1)$, $\tilde{\mat{q}}(\tau+1)=\mat{q}(\tau+1)$, $\tilde{\mat{a}}(\tau+1)=\mat{a}(\tau+1)$ and policy $\mat{M}^{(\tilde{\pi})}(\tau+1)=\mat{M}^{(\pi)}(\tau+1) $. Thus, $\tilde{\mat{x}}(\tau+1) \leq \mat{x}(\tau+1)$ and therefore $\tilde{\mat{x}}(\tau+1) \preceq \mat{x}(\tau+1)$.

\textit{Case 2}: $\tilde{\mat{x}}(\tau)$ is obtained from $\mat{x}(\tau)$ by permutation of two distinct elements $m$ and $n$. In this case, we couple the random variables $\mat{c}(\tau+1)$, $\mat{a}(\tau+1)$ and construct matching $\mat{M}^{(\tilde{\pi})}(\tau+1)$ the same as what we did in Case $2$ of step $2$ in the proof of Lemma \ref{l3}. In addition to these settings, we let $\tilde{q}_{m,k}(\tau+1)=q_{n,k}(\tau+1)$ and $\tilde{q}_{n,k}(\tau+1)=q_{m,k}(\tau+1)$ for $k=1,2,...,K$ and also $\tilde{q}_{i,k}(\tau+1)=q_{i,k}(\tau+1)$ for all $i \in \cal N$, $i\neq n,m$ and $k=1,2,...,K$. By doing such a coupling, we conclude that $\tilde{\mat{x}}(\tau+1)$ and $\mat{x}(\tau+1)$ satisfy condition \textbf{D2} at time slot $\tau+1$ and therefore $\tilde{\mat{x}}(\tau+1) \preceq \mat{x}(\tau+1) $.

\textit{Case 3}: $\tilde{\mat{x}}(\tau)$ is obtained from $\mat{x}(\tau)$ by performing a balancing interchange of two distinct elements $m$ and $n$ as defined in condition \textbf{D3}. In particular, $\tilde{\mat{x}}(\tau)$ and $\mat{x}(\tau)$ are different in only two elements $n$ and $m$ such that $x_n(\tau) <\tilde{x}_n (\tau)\leq\tilde{x}_m (\tau)< x_m(\tau)$ and the following constraints are satisfied; $\tilde{x}_n(\tau)=x_n(\tau)+1$ and $\tilde{x}_m(\tau)=x_m(\tau)-1$.
In this case, with the same argument as what we did in the proof of Lemma \ref{l3}, we can check that $ \tilde{\mat{x}}(\tau+1)$ and $\mat{x}(\tau+1)$ satisfy one of the conditions \textbf{D1}-\textbf{D3}. We omit the details to avoid redundant discussion.

%
%
%
%

\end{proof}

\vspace{-2mm}
\subsection{Statement and Proof of Lemma \ref{lperm}} \label{plperm}
This lemma will be used in the proof of Lemma \ref{lnew}.
\begin{lem} \label{lperm}
Suppose that, at a given time slot $t$, multiple, distinct maximum weighted matchings exist in the graph of Figure \ref{matching-graph}. Any two of them will result in two intermediate queue length vectors which are \textit{equal in permutation}, i.e., one is a permutation of the other.
\end{lem}

\begin{proof}
We need to show that given the queue length vector $\mat{x}(t-1)$ at the beginning of time slot $t$, if we apply two distinct maximum weighted matchings $\mat{M}^{(\text{MWM}1)}(t)$ and $\mat{M}^{(\text{MWM}2)}(t)$ (i.e., when $\mat{M}^{(\text{MWM}1)}(t) \neq \mat{M}^{(\text{MWM}2)}(t)$), then the intermediate queue length vectors $\mat{x}'^{(1)}(t)$ and $\mat{x}'^{(2)}(t)$ resulting from these two matchings are permutations of each other, i.e., for each queue $n \in \mathcal{N}$

 
\begin{itemize}
\item[\textbf{a)}] either queue $n$ is being served by both matchings $\mat{M}^{(\text{MWM}1)}(t)$ and $\mat{M}^{(\text{MWM}2)}(t)$.
\item[\textbf{b)}] or if queue $n$ is being served under $\mat{M}^{(\text{MWM}1)}(t)$ but not under $\mat{M}^{(\text{MWM}2)}(t)$, then there exists a queue $m$ with $x_n(t)=x_{m}(t)$ which is being served under $\mat{M}^{(\text{MWM}2)}(t)$ but not under $\mat{M}^{(\text{MWM}1)}(t)$.
\end{itemize}

We invoke the graph theory analysis we used in the proof of Lemma \ref{l2}. Similarly we define a perfect graph $G'_t$ of size $\max\{N,K\}$ over which we define sub-graphs $G'^{(\text{MWM}1)}_t$ and $G'^{(\text{MWM}2)}_t$ corresponding to $\mat{M}^{(\text{MWM}1)}(t)$ and $\mat{M}^{(\text{MWM}2)}(t)$. We build two directed sub-graphs $D_t^{\text{MWM}1}$ and $D_t^{\text{MWM}2}$ using sub-graphs $G'^{(\text{MWM}1)}_t$ and $G'^{(\text{MWM}2)}_t$ as follows: $D_t^{\text{MWM}1}$ is the same as $G'^{(\text{MWM}1)}_t$ with \textit{positive} edges directed from queues to the servers with positive weights. $D_t^{\text{MWM}2}$ is the same as $G'^{(\text{MWM}2)}_t$ with \textit{negative} edges directed from servers to the queues. Similarly we define graph $U$ as the union of these two sub-graphs, i.e., $U=D^{(\text{MWM}1)}_t \bigcup D^{(\text{MWM}2)}_t$.  

In case ``\textbf{a)}" above, if queue $n$ is being served in both MWM1 and MWM2 then $x'^{(1)}_{n}(t)=x'^{(2)}_n(t)$. Therefore, we consider only case ``\textbf{b)}" and we show that if queue $n\in\cal N$ is being served under $\mat{M}^{(\text{MWM}1)}(t)$ but not under $\mat{M}^{(\text{MWM}2)}(t)$, then there exists a queue $m$ with $x_n(t)=x_{m}(t)$ which is being served under $\mat{M}^{(\text{MWM}2)}(t)$ but not under $\mat{M}^{(\text{MWM}1)}(t)$. We recall the definition of \textbf{Type 0 (T0)}, \textbf{Type 1 (T1)} and \textbf{Type 2 (T2)} edges pairs we had in the proof of Lemma \ref{l2}. If queue $n$ is being served under $\mat{M}^{(\text{MWM}1)}(t)$ but not under $\mat{M}^{(\text{MWM}2)}(t)$, then the edges incident to queue $n$ make a \textbf{T2} pair $r_{n}=(0,x_n(t-1))$. Recall that the graph $U$ is the union of a number of even cycles. Assume that queue $n$ belongs to cycle $\ell$ in graph $U$.
We now trace forward over cycle $\ell$, as shown in Figures \ref{c:1} and \ref{c:2}. The edge pairs after queue $n$ cannot be all \textbf{T0} pairs, since by using the allocations of MWM1 in MWM2 for the queues in $\ell$, we can increase the matching weight index of MWM2 which contradicts the fact that MWM2 is a maximum weighted matching. Thus, we consider the following two cases:
\begin{figure}
  \centering
  	\psfrag{l}[][][1]{$0$}
    \psfrag{m}[][][.9]{$x_n(t-1)$}
    \psfrag{n}[][][.9]{$-x_{m}(t-1)$}
    \psfrag{o}[][][1]{$0$}
    \psfrag{p}[][][1]{$ $}
    \psfrag{q}[][][1]{$ $}
    \psfrag{r}[][][1]{$n$}
    \psfrag{s}[][][1]{$m$}
    \psfrag{t}[][][1]{$T_2$}
    \psfrag{u}[][][1]{$T_0$}
    \psfrag{v}[][][1]{$T_1$}
    \psfrag{a}[][][1]{$0$}
    \psfrag{b}[][][.9]{$x_n(t-1)$}
    \psfrag{d}[][][.9]{$x_{m}(t-1)$}
    \psfrag{c}[][][1]{$0$}
    \psfrag{e}[][][1]{$ $}
    \psfrag{f}[][][1]{$ $}
    \psfrag{g}[][][1]{$n$}
    \psfrag{h}[][][1]{$m$}
    \psfrag{i}[][][1]{$T_2$}
    \psfrag{j}[][][1]{$T_0$}
    \psfrag{k}[][][1]{$T_2$}
  \subfloat[Case 1]{\label{c:1}\includegraphics[width=0.33\textwidth]{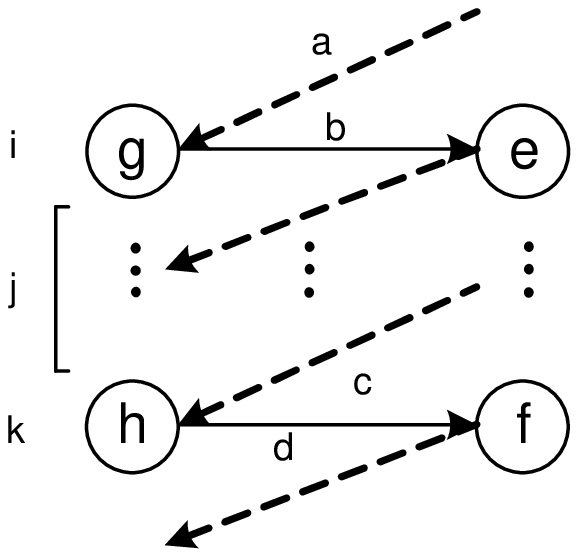}}              
  \subfloat[Case 2]{\label{c:2}\includegraphics[width=0.33\textwidth]{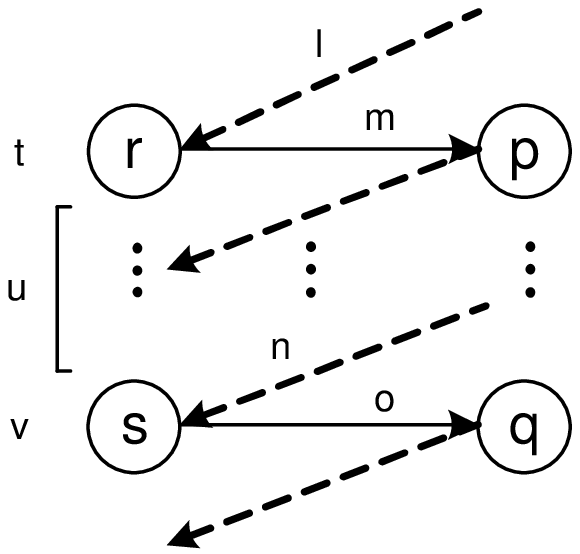}}
  \caption{Cases 1 and 2 in the proof of Lemma \ref{lperm}}\vspace{-7mm}
  \label{examples_lnew}
\end{figure}


\textit{Case 1}: Assume that the first non-\textbf{T0} pair after $r_{n}$ is a \textbf{T2} pair denoted by $r_{m}$. 
This case is shown in Figure \ref{c:1}. In this case, by using the allocations used in MWM1 for queues $n$ to the one right before queue $m$ in cycle $\ell$ in MWM2 and not serving queue $m$, we will obtain a matching whose matching weight index is larger than that of MWM2. This contradicts the fact that MWM2 is a maximum weighted matching.

\textit{Case 2}: Assume that the first non-\textbf{T0} pair after $r_{n}$ is a \textbf{T1} pair denoted by $r_{m}$. This case is shown in Figure \ref{c:2}. In this case, if $x_{m}(t-1)>x_{n}(t-1)$, by using the allocations used in MWM2 for the queues right after $n$ in the cycle $\ell$ to queue $m$ in MWM1 and not serving queue $n$, we obtain a matching whose matching weight index is larger than that of MWM1. This contradicts the fact that MWM1 is a maximum weighted matching. If $x_{m}(t-1)<x_{n}(t-1)$, by using the allocations used in MWM1 for the queues $n$ to the queue right before queue $m$ in cycle $\ell$ in MWM2 and not serving queue $m$, we will obtain a matching whose matching weight index is larger than that of MWM2. This contradicts the fact that MWM2 is a maximum weighted matching. Thus, the only valid case is $x_{m}(t-1)=x_{n}(t-1)$. In this case, either queue $n$ and $m$ are the same (we have covered this case where queue $n$ is served in both MWM1 and MWM2) or queue $n$ and queue $m$ are different in which case the result follows.
\end{proof}

%
%
%
%

\vspace{-2mm}
\subsection{Proof of Lemma \ref{lnew}} \label{plnew}
To show this lemma, we show that any two MWM policy dominate each other, i.e., for any $\pi_1, {\pi}_2 \in \Pi^{\text{MWM}}$, we have $f({\mat{X}}^{(\pi_1)}) \leq_{st} f({\mat{X}}^{(\pi_2)})$ and $f({\mat{X}}^{(\pi_2)}) \leq_{st} f({\mat{X}}^{(\pi_1)}) $. Therefore, according to the definition of ``$\leq_{st}$'', we can conclude that $f({\mat{X}}^{(\pi_1)})$ and $f({\mat{X}}^{(\pi_2)})$ are equal in distribution, i.e., $f({\mat{X}}^{(\pi_1)}) \buildrel \mathcal{D}\over = f({\mat{X}}^{(\pi_2)})$. In order to do so, we will construct a sequence of policies $\tilde{\pi}_1,\tilde{\pi}_2,\dots, \tilde{\pi}_k $ such that $\lim_{k\rightarrow \infty} \tilde{\pi}_k=\pi_2$ and \textit{for all} $k=1,2,\dots$, $f({\mat{X}}^{(\tilde{\pi}_k)}) \buildrel \mathcal{D}\over = f({\mat{X}}^{(\pi_1)})$.

Consider any arbitrary policies $\pi_1, \pi_2 \in \Pi^{\text{MWM}}$. Assume that time slot $t_0$ is the first time slot at which the two policies employ different matching matrices, i.e., we have $\mat{M}^{(\pi_1)}(\tau)=\mat{M}^{(\pi_2)}(\tau)$, $\forall \tau<t_0$ and $\mat{M}^{(\pi_1)}(t_0)\neq \mat{M}^{({\pi}_2)}(t_0)$.
For the system using policy $\pi_1$, for any sample path $\mat{\omega}=(\mat{x}^{(\pi_1)}(0),\mat{c}(1),\mat{a}(1),$ $\mat{x}^{(\pi_1)}(1),\mat{c}(2),\mat{a}(2),\mat{x}^{(\pi_1)}(2),\dots)$ of the underlying random variables $(\mat{X}^{(\pi_1)}(0),\mat{C}(1),\mat{A}(1),\mat{X}^{(\pi_1)}(1),$ $\mat{C}(2),$ $\mat{A}(2),\mat{X}^{(\pi_1)}(2), \dots)$, we can construct a new sample path $\tilde{\mat{\omega}} = ({\mat{x}}^{(\tilde{\pi}_1)}(0), \tilde{\mat{c}}(1), \tilde{\mat{a}}(1), {\mat{x}}^{(\tilde{\pi}_1)}(1),$ $\tilde{\mat{c}}(2), \tilde{\mat{a}}(2),{\mat{x}}^{(\tilde{\pi}_1)}(2), \dots)$ and a new MWM policy $\tilde{\pi}_1$ such that ${\mat{x}}^{(\tilde{\pi}_1)}(t)\buildrel p\over = \mat{x}^{(\pi_1)}(t)$ for all $t= 1,2,\dots$ and $\mat{M}^{(\tilde{\pi}_1)}(\tau)=\mat{M}^{(\pi_2)}(\tau)$, $\forall \tau<t_0+1$. It is important to observe here that policy $\tilde\pi_1$, unlike policy $\pi_1$, uses the same matchings as $\pi_2$ until time $t_0$. The construction of $\tilde\pi_1$ is done as follows:
%
%

\textbf{Construction of $\tilde{\pi}_1$ for $\tau<t_0$}:
We let the arrival, connectivity and the matchings of the new system (which is working under $\tilde{\pi}_1$) be the same as those under $\pi_1$ until time slot $t_0-1$, i.e., $\tilde{\mat{c}}(\tau) =\mat{c}(\tau)$, $\tilde{\mat{a}}(\tau) =\mat{a}(\tau)$ and $\mat{M}^{(\tilde{\pi}_1)}(\tau)=\mat{M}^{(\pi_1)}(\tau)$ for $\tau < t_0$. 
Thus, we have $\mat{x}^{(\tilde{\pi}_1)}(\tau)=\mat{x}^{(\pi_1)}(\tau)$ for $\tau<t_0$. 

\textbf{Construction of $\tilde{\pi}_1$ for $\tau=t_0$}:
At time slot $t_0$, we let $\tilde{\mat{c}}(t_0) =\mat{c}(t_0)$ and $\mat{M}^{(\tilde{\pi}_1)}(t_0) = \mat{M}^{(\pi_2)}(t_0)$. Since $\mat{M}^{(\pi_1)}$ and $\mat{M}^{(\pi_2)}$ are both maximum weighted matchings, according to Lemma \ref{lperm}, the intermediate queue lengths resulted from $\mat{M}^{(\pi_1)}$ and $\mat{M}^{(\pi_2)}$ are permutation of each other, i.e.,
 ${\mat{x}'}^{(\tilde{\pi}_1)}(t_0)=({x'}_1^{(\tilde{\pi}_1)}(t_0), {x'}_2^{(\tilde{\pi}_1)}(t_0),\dots,{x'}_N^{(\tilde{\pi}_1)}(t_0))\buildrel {p}\over = ({x'}_1^{({\pi}_1)}(t_0), {x'}_2^{({\pi}_1)}(t_0),\dots,{x'}_N^{({\pi}_1)}(t_0)) ={\mat{x}'}^{(\pi_1)}(t_0)$. In other words, for each element ${x'}_n^{(\tilde{\pi}_1)}(t_0)$ in ${\mat{x}'}^{(\tilde{\pi}_1)}(t_0)$ there exists an element ${x'}_{m}^{({\pi}_1)}(t_0)$ in ${\mat{x}'}^{({\pi}_1)}(t_0)$ such that ${x'}_n^{(\tilde{\pi}_1)}(t_0)={x'}_{m}^{({\pi}_1)}(t_0)$.
We call queue $m$ and queue $n$ permuted queues. 
In the new system (which is working under $\tilde{\pi}_1$), we let the arrivals of all the permuted queues to be the same at time slot $t_0$, i.e., if queues $n$ and $m$ are two permuted queues we let $\tilde{a}_n(t_0)=a_m(t_0)$.
Thus, for the queue length state of the system at time slot $t_0$, we can easily observe that $\mat{x}^{(\tilde{\pi}_1)}(t_0)\buildrel p\over = \mat{x}^{(\pi_1)}(t_0)$. 

\textbf{Construction of $\tilde{\pi}_1$ for $\tau>t_0$}:
For time slots $\tau > t_0$, we let the connectivities, arrivals and the allocation variables of all the permuted queues to be the same. Thus, if queues $n$ and $m$ are two permuted queues, we let $\tilde{c}_{n,k}(\tau)=c_{m,k}(\tau)$ $\forall k=1,2,\dots,K$, $\tilde{a}_{n}(\tau)=a_{m}(\tau)$ and ${M}_{n,k}^{(\tilde{\pi}_1)}(\tau) ={M}_{m,k}^{({\pi}_1)}(\tau) $ $\forall k=1,2,\dots,K$. Therefore, the service and the evolution of all the permuted queues in the original system and the new system are the same. Since the employed matchings (allocation variables) in policy $\pi_1$ are all maximum weighted matchings, the allocation variables in the new system define a maximum weighted matching.
We can now conclude that the queue length vectors ${\mat{x}}^{(\tilde{\pi}_1)}(\tau)$ and $\mat{x}^{({\pi}_1)}(\tau)$, $\tau>t_0$, are permutations of each other and therefore, $\mat{x}^{(\tilde{\pi}_1)}(\tau)\buildrel p\over = \mat{x}^{(\pi_1)}(\tau)$.

According to Definition \ref{def_order}, we conclude that $f(\mat{x}^{(\pi_1)}(t))=f(\mat{x}^{(\tilde{\pi}_1)}(t))$ for all $t$. Thus, the two policies $\pi_1$ and $\tilde{\pi}_1$ are dominating each other, i.e., $f(\mat{X}^{(\pi_1)})\leq_{st}f(\mat{X}^{(\tilde{\pi}_1)})$ and $ f(\mat{X}^{(\tilde{\pi}_1)})\leq_{st}f(\mat{X}^{(\pi_1)})$ or $f(\mat{X}^{(\pi_1)}) \buildrel \mathcal{D}\over = f(\mat{X}^{(\tilde{\pi}_1)})$. Recall that $\pi_1$ and $\pi_2$ are using similar matchings \textit{until time slot} $t_0-1$. 
We constructed a new MWM policy $\tilde{\pi}_1$ that 
agrees with $\pi_2$ \textit{until time slot} $t_0$ and it still results in the same queue length cost distribution as $\pi_1$, i.e., $f(\mat{X}^{(\tilde{\pi}_1)})\buildrel \mathcal{D} \over = f(\mat{X}^{(\pi_1)})$


By using a mathematical induction approach, we can construct a sequence of policies $\tilde{\pi}_2,$ $\tilde{\pi}_3,\dots,$ whose queue length cost $f(\mat{X})$ is equal in distribution to that of policy $\pi_1$ and are using similar maximum weighted matchings as $\pi_2$ is using at time slots $t_0+1, t_0+2, \dots$. The limiting policy for this sequence of policies is policy $\pi_2$. Therefore, $f(\mat{X}^{(\pi_1)}) \buildrel \mathcal{D}\over = f(\mat{X}^{(\tilde{\pi}_1)})  \buildrel \mathcal{D}\over = f(\mat{X}^{(\tilde{\pi}_2)}) \buildrel \mathcal{D}\over = \dots \buildrel \mathcal{D}\over = f(\mat{X}^{(\pi_2)})$.

\vspace{-5mm}
\bibliographystyle{ieeetran}

\bibliography{Ref}

\end{document}